\documentclass[11pt]{article}

\usepackage[a4paper, hmargin={2.7cm,2.7cm},vmargin={3.3cm,3.3cm}]{geometry}
\usepackage{amsmath,amssymb}
\usepackage{mathtools}
\usepackage{mathrsfs} 
\usepackage{hyperref}
\usepackage{microtype}

\usepackage{amsthm}
\swapnumbers

\newtheorem{theorem}{Theorem}[section]
\newtheorem{corollary}[theorem]{Corollary}
\newtheorem{lemma}[theorem]{Lemma}
\newtheorem{prop}[theorem]{Proposition}

\theoremstyle{definition}
\newtheorem{define}[theorem]{Definition}
\newtheorem{example}[theorem]{Example}
\newtheorem{remark}[theorem]{Remark}

\newcommand{\F}{F}
\newcommand{\T}{T}
\newcommand{\M}{M}
\newcommand{\R}{R}
\newcommand{\related}{$\rho$-related }

\newcommand{\innerprod}{\langle - , - \rangle}
\newcommand{\from}{\colon}
\newcommand{\defeq}{\coloneqq}
\newcommand{\Zd}{\mathbb Z^d}
\newcommand{\Ztd}{\mathbb Z^{2d}}
\newcommand{\Rd}{\mathbb R^d}
\newcommand{\Rtd}{\mathbb R^{2d}}
\newcommand{\MdR}{M(d, \mathbb R)}
\newcommand{\MtdR}{M(2d, \mathbb R)}
\newcommand{\MtdZ}{M(2d, \mathbb Z)}
\newcommand{\GLdR}{\mathrm{GL}(d, \mathbb R)}
\newcommand{\GLtdR}{\mathrm{GL}(2d, \mathbb R)}
\newcommand{\SLtdR}{\mathrm{SL}(2d, \mathbb R)}
\newcommand{\GLtdZ}{\mathrm{GL}(2d, \mathbb Z)}
\newcommand{\SptdR}{\mathrm{Sp}(2d, \mathbb R)}
\newcommand{\UtdR}{U(2d, \mathbb R)}
\newcommand{\MptdR}{\mathrm{Mp}(2d, \mathbb R)}
\newcommand{\MpctdR}{\mathrm{Mp^c}(2d, \mathbb R)}
\newcommand{\Falg}{M^1(\mathbb R^d)}
\newcommand{\blangle}{\bigl\langle}
\newcommand{\brangle}{\bigr\rangle}
\newcommand{\bl}{\bigl}
\newcommand{\br}{\bigr}

\newcommand{\SpthetatdR}{\mathrm{Sp}_\theta(2d, \mathbb R)}
\DeclareMathOperator{\pf}{\mathrm{pf}}
\newcommand\xs{0.25}

\title{On the structure of multivariate Gabor systems and a result on Gaussian Gabor frames}
\author{Michael Gjertsen and Franz Luef}

\begin{document}

\maketitle

\begin{abstract}
    We introduce an equivalence relation on the set of lattices in $\Rtd$ such that equivalent lattices support identical structures of Gabor systems, up to unitary equivalence—a notion we define. These equivalence classes are parameterized by symplectic forms on $\Rtd$ and they consist of lattices related by symplectic transformations. This implies that $2d^2 - d$ parameters suffice to describe the possible structures of Gabor systems over lattices in $\Rtd$, as opposed to the $4d^2$ degrees of freedom in the choice of lattice. We also prove that (modulo a minor complication related to complex conjugation) symplectic transformations are the only linear transformations of the time-frequency plane which implement equivalences of this kind, thereby characterizing symplectic transformations as the structure-preserving transformations of the time-frequency plane in the context of Gabor analysis.
    
    In addition, we investigate the equivalence classes that have separable lattices $A_1 \Zd \times A_2 \Zd$ as representatives and find that the parameter space in this case is $d^2$-dimensional. We provide an explicit example showing that non-separable lattices with irrational lattice points can behave exactly like separable and rational ones. 
    
    This approach also allows us to formulate and prove a higher-dimensional variant of the Lyubarskii-Seip-Wallstén Theorem for Gaussian Gabor frames. This gives us, for a large class of lattices in $\Rtd$ (including all symplectic ones), necessary and sufficient conditions for $d$-parameter families of Gaussians to generate Gabor frames. 
\end{abstract}

\section{Introduction}

Symplectic transformations are ubiquitous in mathematical and theoretical physics. The closely related notion of the metaplectic group and its representation as a group of unitary operators on $L^2(\Rd)$ appears naturally in time-frequency analysis and quantum mechanics, as well as in more classical subjects, such as optics. In time-frequency analysis, symplectic and metaplectic transformations are commonly used to extend the validity of results to more general domains or lattices in the time-frequency plane $\Rtd$. This works because the effects of a symplectic transformation of the time-frequency plane can be counteracted by a metaplectic transformation of $L^2(\Rd)$. This principle is known as \textit{symplectic covariance} and takes the form of Theorem \ref{thm_covariance} in the context of Gabor analysis. We write $\mathcal G(g, \Lambda)$ for the Gabor system generated by an atom $g$ over a lattice $\Lambda$ and $S_{g, g}^\Lambda$ for the associated frame operator on $L^2(\Rd)$, while $\SptdR$ denotes the group of $2d \times 2d$ symplectic matrices over $\mathbb R$. (See Section \ref{prelim_TF_Gabor} for a brief introduction to Gabor analysis and Section \ref{prelim_symplectic} for the symplectic group.)

\begin{theorem}[Symplectic Covariance]
    \label{thm_covariance}
    Let $S \in \SptdR$. Then, there exists a unitary operator $U$ on $L^2(\Rd)$ such that for any lattice $\Lambda \subset \Rtd$ and any atom $g \in L^2(\Rd)$, the Gabor system $\mathcal G(g, \Lambda)$ is a frame if and only if $\mathcal G(Ug, S\Lambda)$ is a frame, and, in this case, their frame operators are related by conjugation by $U$, i.e.\ we have $US^\Lambda_{g, g}U^* = S^{S\Lambda}_{Ug, Ug}$.
\end{theorem}

\noindent The special case $\Lambda = \alpha \Ztd$ ($\alpha > 0$) appears in the now foundational monograph by Gröchenig \cite[Proposition~9.4.4]{Grochenig_book} and the general case has been championed by de Gosson \cite[Proposition~163]{deGosson_book} \cite{deGossonMauriceA.2015HdoG}. Gröchenig uses the special case  to extend results about lattices of the form $\alpha \Zd \times \beta \Zd$ ($\alpha, \beta > 0$) to lattices obtained from these by symplectic transformations, while de Gosson has used it to relate the frame sets of different Gaussians and to investigate deformations of Gabor systems \cite{deGossonMauriceA.2015HdoG}. The role of metaplectic operators in Gabor analysis has also recently been deepened by the work of Cordero and Giacchi \cite{CorderoGiacchi} and Führ and Shafkulovska \cite{FührHartmut2024Tmao}. We are, however, under the impression that the full potential of symplectic covariance is still far from being realized. In this article, we advocate for and extend the power and centrality of this concept for the study of multivariate Gabor systems. Our approach is multifaceted:\
\begin{itemize}
    \item We explain why this is a natural result in Gabor analysis by connecting it to the foundational commutation relations (Section \ref{sec_lattices_and_forms}).
    \item We extend the result by showing that it is compatible with the duality that is central to Gabor analysis (Corollary \ref{corollary_main}).
    \item We reformulate it in a manner that is better suited to concrete computations and applications (Theorem \ref{theorem_main}). At the same time, we explain why much of the subtlety of the metaplectic representation is irrelevant, thus simplifying its presentation (Section \ref{prelim_symplectic} and Appendix \ref{appendix}).
    \item We use it to introduce an equivalence relation on the set of lattices in $\Rtd$ that identifies lattices which support identical structures of Gabor systems. We also parameterize these equivalence classes in a natural way and explain how to easily determine the equivalence class of any given lattice. This gives us a handle on the relevant parameters for the structures of Gabor systems in arbitrary dimensions. (See Definitions \ref{def_symplectically_related} and \ref{definition_main} along with the subsequent discussion.) 
    \item We use this framework to formulate and prove a higher-dimensional variant of the Lyubarskii-Seip-Wallstén Theorem for Gaussian Gabor frames (Theorem \ref{Gaussian_Gabor_frames_v2}).
    \item We prove a converse of symplectic covariance, thereby characterizing symplectic transformations (along with a transformation related to complex conjugation) as precisely those linear transformations of the time-frequency plane which preserve the structure of Gabor systems  (Section \ref{sec_necessity} and Theorem \ref{theorem_converse}).
    \item We further develop the connection between symplectic geometry and Gabor analysis by showing that separable lattices (lattices of the form $A_1\Zd \times A_2 \Zd$) are closely related to the well-known notion of transversal Lagrangian planes (Proposition \ref{prop_separable_lattices_geometric}).
\end{itemize}

Let us begin by turning the conclusion of the theorem under consideration into a definition. Given two lattices $\Lambda$ and $\Lambda'$ in $\Rtd$, we will say that the \textit{structure of Gabor systems} over $\Lambda$ is \textit{unitarily equivalent} to the structure of Gabor systems over $\Lambda'$ if there exists a unitary operator $U$ on $L^2(\Rd)$ such that $\mathcal G(g, \Lambda)$ is a Bessel sequence  if and only if $\mathcal G(Ug, \Lambda')$ is a Bessel sequence, and, moreover, the (mixed-type) frame operators are related by conjugation:\
\begin{align*}
    S_{Ug, Uh}^{\Lambda'} = U S^\Lambda_{g, h} U^* \quad \text{whenever } \mathcal G(g, \Lambda) \text{ and } \mathcal G(h, \Lambda) \text{ are Bessel sequences.}
\end{align*}
This further implies that $\mathcal G(g, \Lambda)$ is a frame if and only if $\mathcal G(Ug, \Lambda')$ is a frame and that their optimal frame bounds are equal. See Definition \ref{definition_main} and the subsequent discussion for details. ``The structure of Gabor systems over $\Lambda$'' quickly becomes a mouthful, so we abbreviate this further by saying that the \textit{Gabor structures} over $\Lambda$ and $\Lambda'$ are unitarily equivalent.

If we choose lattice matrices for $\Lambda$ and $\Lambda'$, say $\Lambda = A \Ztd$ and $\Lambda' = B \Ztd$, we will see that the Gabor structures over $\Lambda$ and $\Lambda'$ are unitarily equivalent whenever $A^T J A = B^T J B$, where $J$ is the standard symplectic matrix defined in Equation \eqref{std_symp_matrix}. This gives us a sufficient condition for unitary equivalence that is exceedingly simple to check. This is the content of Theorem \ref{theorem_main}. Moreover, as we explain in Appendix \ref{appendix}, one can often write down a unitary operator which implements this equivalence explicitly:\ one can take it to be any metaplectic operator associated to the (guaranteed to be symplectic) matrix $AB^{-1}$. We will also see that it is not really necessary to invoke the full abstract machinery of the metaplectic group.

Theorem \ref{theorem_main} implies that the entries of the matrix $A^T J A$ uniquely determine the structure of Gabor systems over $A \Ztd$. Matrices of the form $A^T J A$ with $A \in \GLtdR$\footnote{See the beginning of Section \ref{sec_prelims} for our notation regarding spaces of matrices.} are precisely the invertible and antisymmetric matrices. This means that, although the space $\GLtdR$ of all lattice matrices is $(2d)^2$-dimensional, the possible structures of Gabor systems depend upon at most $(2d)(2d-1)/2 = 2d^2 - d$ parameters. In Section \ref{sec_geometric_content}, we relate these parameters to the geometry of the lattice $A\Ztd$, and we explain how they appear as the central quantities in the multivariate extension, due to Bourouihiya \cite{Bourouihiya}, of the famous Lyubarskii-Seip-Wallstén result on Gaussian Gabor frames. We then generalize this result in Theorem \ref{Gaussian_Gabor_frames_v2}. We also explore the case of separable lattices, first from a computational point of view in Section \ref{sec_separable_lattices} and then from a geometrical point of view in Section \ref{subsec_separable_geometric}.

As an illustration of the power of these results, consider the lattice $A \mathbb Z^4$, where
\begin{align*}
    A = \begin{pmatrix}
        1 & \pi^2 & \frac{\pi}3 & \pi
        \\[\xs em] \pi^2 & 1 & \pi & \frac{\pi}2 
        \\[\xs em] 0 & \pi & \frac13 & 1
        \\[\xs em] \pi & 0 & 1 & \frac12
    \end{pmatrix}.
\end{align*}
We will show that the simple calculation
\begin{align*}
    A^T J A = \begin{pmatrix}
        0 & 0 & \frac13 & 1
        \\[\xs em] 0 & 0 & 1 & \frac12 
        \\[\xs em] -\frac13 & -1 & 0 & 0
        \\[\xs em] -1 & -\frac12 & 0 & 0
    \end{pmatrix}
\end{align*}
implies that the structure of Gabor systems over $A\mathbb  Z^4$ is unitarily equivalent to the structure of Gabor systems over the separable and rational (!) lattice
\begin{align*}
    \begin{pmatrix}
        1 & 0 & 0 & 0 
        \\[\xs em] 0 & 1 & 0 & 0 
        \\[\xs em] 0 & 0 & \frac13 & 1 
        \\[\xs em] 0 & 0 & 1 & \frac12
    \end{pmatrix} \mathbb Z^4.
\end{align*}
Moreover, by another simple calculation involving these matrices, we can immediately conclude that the equivalence is implemented by the unitary operator $U$ defined by
\begin{align*}
    Uf(t) = \frac{1}{\pi} \int_{\mathbb R^2} f(x) e^{i(t \cdot t - 2 x \cdot t + x \cdot Mx)} \,dx, \quad \text{where } M = \begin{pmatrix}
        1 & \pi^2 \\ \pi^2 & 1
    \end{pmatrix},
\end{align*}
with $f \in L^2(\mathbb R^2)$ and $t \in \mathbb R^2$. The details can be found in Example \ref{example_main}.

As another example, consider the lattice $B \mathbb Z^4$, where
\begin{align*}
    B = \begin{pmatrix}
        3 & 0 & \frac12 & 0
        \\[\xs em] 0 & 1 & 0 & 1 
        \\[\xs em] 14 & 1 & \frac52 & 1
        \\[\xs em] 3 & \frac{14}3 & \frac12 & 5
    \end{pmatrix}.
\end{align*}
By our higher-dimensional variant of the Lyubarskii-Seip-Wallstén theorem, the simple calculation 
\begin{align*}
    B^T J B = \begin{pmatrix}
        0 & 0 & \frac12 & 0
        \\[\xs em] 0 & 0 &  0 & \frac13 
        \\[\xs em] -\frac12 & 0 & 0 & 0
        \\[\xs em] 0 & -\frac13 & 0 & 0
    \end{pmatrix}
\end{align*}
and the fact that $\tfrac12, \tfrac13 \in (-1, 1)$, implies that, for any nonzero $a, b \in \mathbb R$, the Gaussian
\begin{align*}
    g(t) = \exp \left( - \pi t^T \begin{pmatrix}
        \frac{a^2 - i (42a^4 + 5) }{9a^4+1} & -i 
        \\[\xs em] -i & \frac{3b^2 - i (14 b^4 + 135) }{3(b^4 + 9)}
    \end{pmatrix} t \right)
\end{align*}
generates a Gabor frame over $B \mathbb Z^4$. For more details, see Example \ref{Gaussian_example}.

Here are two other examples of lattices in $\mathbb Z^4$ (from Example \ref{Gaussian_example_2}) for which we get two-parameter families of Gaussian Gabor frames:\ $C\mathbb Z^4$ and $D\mathbb Z^4$, where 
\begin{align*}
    C = \begin{pmatrix}
        -\frac{14}{3} & -\frac{4}{21} & \frac{1}{25} & \frac{2}{5} 
        \\[\xs em] \frac{16}{3} & \frac{1}{21} & -\frac{2}{25} & \frac{1}{5} 
        \\[\xs em] -\frac{14}{3} & -\frac{11}{24} & \frac{1}{100} & \frac{19}{40} 
        \\[\xs em] \frac{91}{12} & -\frac{1}{6} & -\frac{13}{200} & \frac{7}{20}
    \end{pmatrix} \quad \text{and} \quad D = \begin{pmatrix}
        2\pi & -\frac{17}{9} & -\frac{2}{65} & \frac{3\sqrt{2}}{13} 
        \\[\xs em] \frac{\pi}{3} & -2 & \frac{3}{65} & \frac{2\sqrt{2}}{13}
        \\[\xs em] \frac{63\pi}{8} & -\frac{47}{12} & -\frac{57}{520} & \frac{33}{26\sqrt{2}} 
        \\[\xs em] -\frac{11\pi}{4} & -\frac{37}{8} & \frac{3}{20} & \frac{3}{4\sqrt{2}}
    \end{pmatrix}.
\end{align*}
As before, the proof of this amounts to the simple calculations
\begin{align*}
    C^T J C = \begin{pmatrix}
        0 & 0 & \frac25 & 0
        \\[\xs em] 0 & 0 &  0 & \frac17 
        \\[\xs em] -\frac25 & 0 & 0 & 0
        \\[\xs em] 0 & -\frac17 & 0 & 0
    \end{pmatrix} \quad \text{and}  \quad D^T J D = \begin{pmatrix}
        0 & 0 & \frac{\pi}5 & 0
        \\[\xs em] 0 & 0 &  0 & -\frac{\sqrt{2}}{3} 
        \\[\xs em] -\frac{\pi}5 & 0 & 0 & 0
        \\[\xs em] 0 & \frac{\sqrt{2}}{3} & 0 & 0
    \end{pmatrix}
\end{align*}
and the observation that $\tfrac{2}{5}, \tfrac{1}{7}, \tfrac{\pi}{5}, - \tfrac{\sqrt{2}}{3} \in (-1, 1)$. We can write down the Gaussians explicitly in these cases as well, but the expressions would not be pretty.

We emphasize that there is nothing special about $\mathbb Z^4$; the results illustrated by these examples hold in any dimension. For a lattice in $\Ztd$, we obtain $d$-parameter families of Gaussians. Recently, multivariate Gaussian Gabor frames have been used in work related to optimal sphere packings \cite{manin_marcolli} and neuroscience \cite{liontou_marcolli_1,liontou_marcolli_2}. We plan to investigate the relevance of our results on Gaussian Gabor frames for these problems.

This connection between lattices and symplectic forms appears quite explicitly in the currently unfolding connection between Gabor analysis and modules over noncommutative tori \cite{Franz_projections_1, Franz_projections_2, Jakobsen_Franz, Austad_Enstad}. The noncommutative tori that appear in Gabor analysis are parameterized precisely by symplectic forms (in the guise of invertible and antisymmetric matrices):\ for each lattice $A \Ztd$, we obtain an action of the noncommutative torus determined by the symplectic form $A^T J A$ on (a completion of) the Feichtinger algebra on $\Rd$. In this context, the fact that the structure of Gabor systems over $A\Ztd$ is determined by the symplectic form $A^T J A$ can be understood in terms of the structure of projections in the associated noncommutative torus. We will not pursue this connection here, but we mention it as a potentially fruitful direction for further research. In collaboration with Chakraborty, we have begun investigating the role of metaplectic operators in this setting in \cite{Chakraborty_Franz, masters}. The latter reference contains a basic introduction to the subject.

Before we proceed, we wish to thank the anonymous reviewer for very thorough and valuable feedback—the article has been significantly restructured based on the suggestions.

\section{Preliminaries}
\label{sec_prelims}

In this section, we give brief and selective introductions to time-frequency analysis, Gabor analysis, symplectic linear algebra and metaplectic operators. For much more complete treatments of these topics, see Gröchenig \cite{Grochenig_book} and de Gosson \cite{deGosson_book}. Readers who are already familiar with metaplectic operators should nevertheless have a look at Definition \ref{rho_related}, where we define $\rho$-relatedness. This is a term we introduce to bypass the subtleties of the metaplectic group. 

Let us briefly establish some basic notation and terminology. We will write $\MtdR$ for the set of all $2d \times 2d$ matrices with entries in $\mathbb R$, $\GLtdR$ for the subset (and group) of invertible matrices and $\SLtdR$ for the subgroup of invertible matrices with determinant equal to one. We will freely use variants of this notation, such as $\MdR$, $\GLdR$, $\MtdZ$ and $\GLtdZ$. We require elements of $\GLtdZ$ to be invertible in $\MtdZ$, so that an element of $\GLtdZ$ is a $2d \times 2d$ invertible matrix with integer entries whose inverse also has only integer entries. A matrix $M$ is called \textit{antisymmetric/skew symmetric} if $M^T = -M$. We will write $\cdot$ and $\innerprod$ for the standard inner products on $\Rd$ and $L^2(\Rd)$, respectively—we take the latter to be linear in the first argument. The term \textit{phase factor} refers to a complex number of unit length, i.e.\ $\alpha \in \mathbb C$ with $|\alpha| = 1$.

\subsection{Time-Frequency Analysis and Gabor Frames}
\label{prelim_TF_Gabor}

For $x, \omega \in \Rd$, we define the \textit{translation} and \textit{modulation} operators $T_x$ and $M_\omega$ on $L^2(\Rd)$ by
\begin{align*}
    T_x f(t) = f(t-x) \quad \text{and} \quad M_\omega f(t) = e^{2\pi i t \cdot \omega} f(t),
\end{align*}
where $f \in L^2(\Rd)$. For $z = (x, \omega) \in \Rtd$, the composition $\pi(z) \coloneqq M_\omega T_x$ is referred to as a \textit{time-frequency shift}. We will also work with \textit{symmetrized time-frequency shifts},
\begin{align*}
    \rho(z) \coloneqq e^{- \pi i x \cdot \omega} \pi(z) = (M_{\omega/2} T_{x/2})(T_{x/2}M_{\omega/2}) = (T_{x/2}M_{\omega/2})(M_{\omega/2} T_{x/2}).
\end{align*}
These are also known as \textit{Heisenberg-Weyl operators}. The equality of the various expressions for $\rho(z)$ follows from the relation
\begin{align*}
    T_x M_\omega = e^{-2 \pi i x \cdot \omega} M_\omega T_x,
\end{align*}
which is easily verified by evaluating both sides at an arbitrary $f \in L^2(\Rd)$. This relation further implies that, for $z = (x, \omega), w = (y, \eta) \in \Rtd$, we have
\begin{equation}
\begin{aligned}
    \pi(z) \pi(w) &= e^{2 \pi i (y \cdot \omega - x \cdot \eta)} \pi(w) \pi(z) \\
    \text{and} \quad  \rho(z) \rho(w) &= e^{2 \pi i (y \cdot \omega - x \cdot \eta)} \rho(w) \rho(z).
\end{aligned}
\label{commutation_relations}
\end{equation}
These identities are central to Gabor analysis—we refer to them as \textit{the commutation relations}. If we define
\begin{align}
    \label{std_symp_matrix}
    J \defeq \begin{pmatrix} 0 & I \\ -I & 0 \end{pmatrix},
\end{align}
where $I$ denotes the $d \times d$ identity matrix, then we can write the commutation relations as
\begin{align}
    \label{std_commutation_relation_with_J}
    \pi(z) \pi(w) = e^{2 \pi i w^T J z} \pi(w) \pi(z) \quad \text{for } z, w \in \Rtd,
\end{align}
and likewise for the symmetrized time-frequency shifts. The following properties of $J$ are essential and will frequently be used without mention:\
\begin{align*}
    J^2 = -I \quad \text{and} \quad J^T = J^{-1} = -J
\end{align*}
(where $I$ now denotes the $2d \times 2d$ identity matrix).

In Gabor analysis, we fix a \textit{lattice} $\Lambda \coloneqq A \Ztd$, where $A \in \GLtdR$, an \textit{atom} $g \in L^2(\Rd)$ and consider the \textit{Gabor system} 
\begin{align*}
    \mathcal G(g, \Lambda) \coloneqq \{ \pi(z) g : z \in \Lambda \} = \{ \pi(Ak) g : k \in \Ztd \}.
\end{align*}
In this context, we call $A$ a \textit{lattice matrix} for $\Lambda$ and say that $A$ \textit{determines} the lattice $\Lambda$. Some of our results are best formulated by emphasizing lattice matrices (as opposed to lattices), so we find it convenient to also introduce the notation $\mathcal G(g, A) \coloneqq \mathcal G(g, A\Ztd)$. 

Any given lattice $\Lambda \subset \Rtd$ has multiple distinct lattice matrices. The following remark contains a precise characterization of the redundancy.

\begin{remark}[Redundancy of Lattice Matrices]
    \label{remark_redundancy_of_lattice_matrices}
    For any $A, B \in \GLtdR,$ the following statements are equivalent. 
    \begin{itemize}
        \item[(a)] $A$ and $B$ determine the same lattice:\ $A \Ztd = B \Ztd$.
        \item[(b)] $A^{-1}B \in \GLtdZ$, i.e.\ $B = A \M$ for some $\M \in \GLtdZ$.
    \end{itemize}
    In particular, if $\Lambda \subset \Rtd$ is a lattice and $A$ is any lattice matrix for $\Lambda$, then the set of all lattice matrices for $\Lambda$ is given by $A \GLtdZ$. All of this follows from the equivalences
    \begin{align*}
        A \Ztd = B \Ztd \iff \Ztd = A^{-1}B \Ztd \iff A^{-1}B \in \GLtdZ.
    \end{align*}
    Geometrically, $\Ztd \mapsto \M \Ztd$ represents an additive permutation of $\Ztd$, and so $\Lambda = A\Ztd \mapsto A \M \Ztd$ represents an additive permutation of $\Lambda$, i.e.\ an automorphism of $\Lambda$ as an abelian group.
\end{remark}

The central question in Gabor analysis is the following:\ for which Gabor systems $\mathcal G(g, A)$ can we find constants $C_1, C_2 > 0$ such that 
\begin{align*}
    C_1 \|f\|^2 \leq \sum_{k \in \Ztd} | \langle f, \pi(Ak) g \rangle |^2 \leq C_2 \|f\|^2 \quad \text{for all } f \in L^2(\Rd)?
\end{align*}
A Gabor system $\mathcal G(g, A)$ for which the upper bound holds is called a \textit{Bessel sequence}, and if the lower bound holds as well, we call it a (\textit{Gabor}) \textit{frame}. For a Bessel sequence, the optimal choice of $C_2$ is called the \textit{optimal Bessel bound}, and the optimal choices of $C_1$ and $C_2$ for a frame are called the \textit{optimal frame bounds}. Note that we always have $|\langle f, \pi(Ak) g \rangle| = |\langle f, \rho(Ak) g \rangle|$, so that nothing of essence changes if we replace time-frequency shifts with their symmetrized versions.

For a Bessel sequence $\mathcal G(g, A)$, we can introduce the operators
\begin{alignat*}{3}
\begin{split}
C_g^A \from L^2(\Rd) &\to \ell^2(\Ztd) \\ \phantom{\sum\nolimits_k} f &\mapsto \bl( \langle f, \pi(Ak) g \rangle \br)_k 
\end{split} &\qquad &\text{and} &\qquad \begin{split} D_g^A \from \ell^2(\Ztd) &\to L^2(\Rd) \\ (a_k)_k &\mapsto \sum\nolimits_k a_k \pi(Ak) g, \end{split}
\end{alignat*}
referred to as \textit{analysis} and \textit{synthesis operators}, respectively. The sums defining $D_g^A$ can be shown to converge unconditionally. Given two Bessel sequences $\mathcal G(g, A)$ and $\mathcal G(h, A)$, the composition $S_{g, h}^A \coloneqq D_h^A \circ C_g^A$ is called a (\textit{mixed-type}) \textit{frame operator}. Explicitly, $S_{g, h}^A$ is the operator
\begin{align*}
    S_{g, h}^A \from L^2(\Rd) &\to L^2(\Rd) \\
    f &\mapsto S_{g, h}^A(f) = \sum_{k \in \Ztd} \langle f, \pi(Ak) g \rangle \pi(Ak) h.
\end{align*}
Exchanging time-frequency shifts for their symmetrized versions does have an effect on the synthesis and analysis operators, but the additional phase factors that appear will cancel out in the composition $D_h^A \circ C_g^A$, so the frame operators are unchanged. If we wish to emphasize the lattice $\Lambda = A \Ztd$ instead of the lattice matrix $A$, we write $S_{g, h}^\Lambda$ instead of $S_{g, h}^A$.

The foremost feature that makes frames so attractive is that, given any Gabor frame $\mathcal G(g, A)$, there exists a (generally non-unique) Bessel sequence $\mathcal G(h, A)$ such that $S^A_{g, h} = I = S^A_{h, g}$, meaning that
\begin{align*}
    f = \sum_{k \in \Ztd} \langle f, \pi(Ak)g \rangle \pi(Ak)h = \sum_{k \in \Ztd} \langle f, \pi(Ak)h \rangle \pi(Ak)g \quad \text{for all } f \in L^2(\Rd),
\end{align*}
with unconditional convergence in norm. In this situation, $\mathcal G(h, A)$ is automatically also a frame, and we call it a \textit{dual frame} for $\mathcal G(g, A)$. The atom $h$ is called a \textit{dual atom} of $g$ (over $\Lambda = A \Ztd$). Although the dual atom is not unique, there is a canonical ``optimal choice'', namely $h = (S^A_{g, g})^{-1} g$, which is called the \textit{canonical dual atom} of $g$. This choice is optimal in the sense that it minimizes the $\ell^2$-norms of the coefficients $\langle f, \pi(Ak) h \rangle$ in the expansions above.

The central result in Gabor analysis is a duality theorem \cite{Koppensteiner}. In order to state it, we need to introduce the notions of adjoint lattices and Riesz basic sequences. A Gabor system $\mathcal G(g, A)$ is called a \textit{Riesz basic sequence} if there exist constants $C_1, C_2 > 0$ such that 
\begin{align*}
    C_1 \| a \|_2 \leq \Bigl\| \sum_{k \in \Ztd} a_k \pi(Ak) g \Bigr\| \leq C_2 \| a \|_2 \quad \text{for all } a = (a_k)_k \in \ell^2(\Ztd).
\end{align*}
Given a lattice $\Lambda = A \Ztd$, we define its \textit{adjoint lattice} 
\begin{align*}
    \Lambda^\circ \coloneqq \{z \in \Rtd : \pi(z)\pi(Ak) = \pi(Ak)\pi(z) \text{ for all } k \in \Ztd \}.
\end{align*}
Using the commutation relations, one finds that $\Lambda^\circ = -JA^{-T} \Ztd$. We find it convenient to introduce the notation $A^\circ \coloneqq -JA^{-T}$, so that $\Lambda^\circ = A^\circ \Ztd$. Note that
\begin{align*}
    (A^\circ)^\circ = -J(-JA^{-T})^{-T} = (-J)(-J)^{-T} A = - A,
\end{align*}
so that $(\Lambda^\circ)^\circ = \Lambda$. The size of a \textit{fundamental domain} $A[0, 1)^{2d}$ of our lattice plays a crucial role, so we denote it by $|A| \coloneqq |\det A|$ and call it the \textit{covolume} of the lattice $A\Ztd$. Note also that $|A^\circ|= |A|^{-1}$.

\begin{theorem}[The Duality Theorem]
    \label{theorem_duality}
    Let $A \in \GLtdR$ and $g \in L^2(\Rd)$. Then, $\mathcal G(g, A)$ is a Bessel sequence if and only if $\mathcal G(g, A^\circ)$ is a Bessel sequence. Moreover, the following statements are equivalent:
    \begin{itemize}
        \item[(a)] $\mathcal G(g, A)$ is a Gabor frame;
        \item[(b)] $\mathcal G(g, A^\circ)$ is a Bessel sequence and there exists a Bessel sequence $\mathcal G(h, A^\circ)$ such that
        \begin{align*}
            \langle h, \pi(A^\circ k) g \rangle = |A| \delta_{k0} \quad \text{for all } k \in \Ztd;
        \end{align*}
        \item[(c)] $\mathcal G(g, A^\circ)$ is a Riesz basic sequence.
    \end{itemize}
\end{theorem}

\noindent Condition $(b)$ in the duality theorem is referred to as the \textit{Wexler-Raz biorthogonality relations}, and it turns out to characterize dual atoms. Here is a precise statement.

\begin{prop}[The Wexler-Raz Biorthogonality Relations]
    \label{prop_Wexler}
    Let $A \in \GLtdR$ and $g, h \in L^2(\Rd)$ be such that $\mathcal G(g, A)$ and $\mathcal G(h, A)$ are Bessel sequences. Then, $S_{g, h}^A = I = S_{h, g}^A$ (i.e.\ $g$ and $h$ are dual atoms over $A\Ztd$) if and only if 
    \begin{align*}
        \langle \pi(A^\circ k) h, \pi(A^\circ l) g \rangle = |A| \delta_{kl} \quad \text{for all } k, l \in \Ztd.
    \end{align*}
\end{prop}

\noindent As a consequence of duality, one obtains the following necessary condition for lattices to support Gabor frames.

\begin{theorem}[The Density Theorem]
    \label{theorem_density}
    Let $A \in \GLtdR$. If there exists any $g \in L^2(\Rd)$ such that $\mathcal G(g, A)$ is a frame, then $0 < |A| \leq 1$.
\end{theorem}

\noindent Proofs of the duality theorem, the Wexler-Raz biorthogonality relations and the density theorem can all be found in Gröchenig and Koppensteiner \cite{Koppensteiner}.

The density theorem shows that covolume plays a crucial role in the structure of Gabor systems. In Section \ref{sec_geometric_content}, we will see that the covolume of $A\Ztd$ has a surprisingly simple relation to the entries of $A^T J A$.

\subsection{The Symplectic Group and the $\rho$-Related Unitary Operators}
\label{prelim_symplectic}

Consider the bilinear form $\sigma \from \Rtd \times \Rtd \to \mathbb R$ represented by $J$ (as defined by Equation \eqref{std_symp_matrix}), namely
\begin{align}
    \label{std_symp_form}
    \sigma(z, w) = w^T J z = y \cdot \omega - x \cdot \eta \quad \text{with } z = (x, \omega), w = (y, \eta) \in \Rtd.
\end{align}
This is the \textit{standard symplectic form} on $\Rtd$. We have already seen it appear in the commutation relations \eqref{std_commutation_relation_with_J}.

The \textit{symplectic group} $\SptdR$ consists of all matrices that preserve the standard symplectic form. In other words, for any $2d \times 2d$ matrix $S$, we have
\begin{align*}
    S \in \SptdR &\iff \sigma(Sz, Sw) = \sigma(z, w) \text{ for all } z, w \in \Rtd \\ &\iff S^TJS = J.
\end{align*}
It turns out that $\SptdR$ is a subgroup of $\SLtdR$ \cite[Section~2.3.3]{deGosson_book}. Note that $J$ itself is a symplectic matrix, so it plays two distinct roles in this subject. (This is analogous to the fact that the identity matrix $I$ both represents the standard inner product on $\Rd$ and is an orthogonal matrix.)

The standard symplectic form $\sigma$ has the following properties:
\begin{itemize}
    \item[(i)] $\sigma(z, w) = - \sigma(w, z)$ for all $z, w \in \Rtd$ (\textit{anti-symmetry});
    \item[(ii)] if $\sigma(z, w) = 0$ for all $w \in \Rtd$, then $z = 0$ (\textit{non-degeneracy}).
\end{itemize}
Any bilinear form on $\Rtd$ with these two properties is called a \textit{symplectic form}. It is not difficult to see that they are precisely the bilinear forms represented by antisymmetric and invertible matrices. That is, for $\theta \in \MtdR$, the bilinear form $\Omega \from \Rtd \times \Rtd \to \mathbb R$ defined by
\begin{align*}
    \Omega(z, w) = w^T \theta z \quad \text{for } z, w \in \Rtd
\end{align*}
is symplectic if and only if $\theta$ is invertible and antisymmetric ($\theta^T = - \theta$). In this situation, we say that $\Omega$ is \textit{represented} by $\theta$. We will blur the distinction between a symplectic form and the matrix that represents it; we will mainly work with matrices and refer to antisymmetric and invertible matrices $\theta$ as symplectic forms.

Using a variant of the Gram-Schmidt process, one can prove that any symplectic form is related to the standard one by a change of basis in $\Rtd$ \cite[Section~1.1]{SilvaAna}. That is, given any symplectic form $\Omega$, we can find $A \in \GLtdR$ such that
\begin{align}
    \label{eq_symp_form_determined_by_A}
    \Omega(z, w) = \sigma(Az, Aw) \quad \text{for all } z, w \in \Rtd.
\end{align}
If $\theta$ represents $\Omega$, this is equivalent to $A^T J A = \theta$. In this context, we say that $A$ \textit{determines} the symplectic form $A^T J A$. As with lattice matrices, these are not unique, but we can give a precise description of the redundancy. This simple observation is crucial to our classification scheme for lattices.  

\begin{lemma}
    \label{lemma_symplectic_conjugation}
    For $A, B \in \GLtdR$, the following conditions are equivalent.
    \begin{itemize}
        \item[(a)] $A$ and $B$ determine the same symplectic form:\ $A^T J A = B^T J B$.
        \item[(b)] $BA^{-1} \in \SptdR$, i.e.\ $B = SA$ for some $S \in \SptdR$.
    \end{itemize}
    In particular, if $\theta$ is a symplectic form and $A$ is any matrix that determines $\theta$, then the set of all matrices that determine $\theta$ is given by $\SptdR A$.
    \end{lemma}
\begin{proof} 
    We have $A^T JA = B^T JB$ if and only if $J = (BA^{-1})^T J (BA^{-1})$.
\end{proof}

\noindent Note that the roles of $A$ and $B$ are symmetric, so that $A^T JA = B^T J B$ is equivalent to $AB^{-1} \in \SptdR$ as well. This just amounts to the trivial observation that if $B = SA$ with $S \in \SptdR$, then $A = S^{-1} B$ and $S^{-1} \in \SptdR$. 

It turns out that for every $S \in \SptdR$, we can find a unitary operator $U$ on $L^2(\mathbb R^d)$ with the property that 
\begin{align}
    \label{metaplectic_op}
    U \rho(z)U^* = \rho(Sz) \quad \text{for all } z \in \Rtd.
\end{align}
This identity does not determine $U$ uniquely:\ if $U$ satisfies  \eqref{metaplectic_op}, then so does $\alpha U$ for any phase factor $\alpha$ (i.e.\ $\alpha \in \mathbb C$ with $|\alpha| = 1$).

\begin{prop}
    \label{prop_metaplectic_ops}
    For every $S \in \SptdR$, there exists a unitary operator $U$ on $L^2(\Rd)$ such that \eqref{metaplectic_op} is satisfied. Moreover, any bounded and invertible operator $\F$ on $L^2(\Rd)$ satisfying 
    \begin{align}
        \label{intertwining_op}
        \F \rho(z) \F^{-1} = \rho(Sz) \quad \text{for all } z \in \Rtd
    \end{align}
    is of the form $\F = \alpha U$ with $\alpha \in \mathbb C$. If $F$ is unitary, then $|\alpha| = 1$.
\end{prop}
\begin{proof}
    For the proof of existence, see Appendix \ref{appendix}. The fact that $\F$ is determined uniquely up to scalar multiples follows from the fact that \eqref{intertwining_op} exhibits $\F$ as an intertwining operator between two irreducible unitary representations of the Heisenberg group. See Folland \cite[Sections~3.1~and~6.7]{Folland} for details. (This claim of uniqueness will not matter for any of our results.)
\end{proof}

The following definition is crucial; it allow us to circumvent the subtle construction of the metaplectic group (to be discussed shortly). 

\begin{define}[$\rho$-Relatedness]
    \label{rho_related}
    Given a symplectic matrix $S$, we will say that a unitary operator $U$ on $L^2(\Rd)$ is \textit{\related}to $S$ if \eqref{metaplectic_op} is satisfied, i.e.\ if $U \rho(z) U^* = \rho(Sz)$ for all $z \in \Rtd$. 
\end{define}

\noindent Note that if two unitary operators $U_1$ and $U_2$ are \related to $S_1$ and $S_2$ (respectively), then $U_1 U_2$ is \related to $S_1 S_2$, since
\begin{align*}
    (U_1 U_2) \rho(z) (U_1 U_2)^* = U_1 \rho(S_2 z) U_1^* = \rho(S_1 S_2 z).
\end{align*}

At this point, it is common to introduce the \textit{metaplectic group} $\MptdR$. It turns out that for each $S \in \SptdR$, we can choose precisely two unitary operators that are \related to $S$ in such a manner that the resulting collection of operators—denoted by $\MptdR$—forms a group under composition, and we obtain a two-to-one group homomorphism $\MptdR \to \SptdR$. Since all the unitary operators that are \related to $S$ differ only by phase factors (by Proposition \ref{prop_metaplectic_ops}), the metaplectic group arises from careful choices of phase factors. Keeping track of the phase factors needed to obtain the metaplectic group will serve us no purpose and only complicate our exposition, so we see no reason to do so. For those already familiar with the metaplectic group, we remark that ignoring the phase factors amounts to working with the \textit{circle extension} of the metaplectic group, namely 
\begin{align*}
    \MpctdR \coloneqq \{ \alpha U \mid \alpha \in \mathbb C, |\alpha| = 1 \text{ and } U \in \MptdR \}.
\end{align*}
Whenever we speak of a unitary operator $U$ that is \related to $S \in \SptdR$, one is of course free to take $U \in \MptdR$, but $U$ can be any unitary operator satisfying \eqref{metaplectic_op} (and by the uniqueness result of Proposition \ref{prop_metaplectic_ops}, any such operator differs from a metaplectic one only by a phase factor, i.e.\ $U \in \MpctdR$).

In Appendix \ref{appendix}, we explicitly construct the \related unitary operators for a large class of symplectic matrices. In particular, we identify a simple set of generators for the symplectic group and find their \related unitary operators. We emphasize that one can do all of this without ever worrying about choices of phase factors—straightforward calculations and linear algebra suffices. We will use these generators and operators in some of our calculations, so we reproduce the result here. 

Let $P = P^T \in \MdR$ and $L \in \GLdR$. Define the $2d \times 2d$ matrices
\begin{align*}
    V_P \coloneqq \begin{pmatrix}
        I & 0 \\ -P & I 
    \end{pmatrix} \quad \text{and} \quad M_L \coloneqq \begin{pmatrix}
        L^{-1} & 0 \\ 0 & L^T
    \end{pmatrix}.
\end{align*}
A straightforward calculation shows that these are symplectic. We also define unitary operators $U_{V_P}$ and $U_{M_L}$ on $L^2(\Rd)$ by
\begin{align}
    \label{def_of_U_V_and_U_M}
    U_{V_P} f(t) = e^{-\pi i (Pt)\cdot t} f(t) \quad \text{and} \quad U_{M_L} f(t) = \sqrt{|\det L|} f(Lt)
\end{align}
for all $f \in L^2(\Rd)$ and $t \in \Rd$. Moreover, with an eye to the standard symplectic matrix $J$, we denote the Fourier transform $\mathcal F$ on $L^2(\Rd)$ by $U_J$, i.e.\
\begin{align*}
    U_J f(t) = \mathcal Ff(t)  = \int_{\Rd} f(x) e^{-2 \pi i x \cdot t} \, dx.
\end{align*}

\begin{prop}[Generators of the Symplectic Group]
    \label{prop_generators}
    The collection 
    \begin{align*}
        \{ J \} \cup \{ V_P : P = P^T \in \MdR \} \cup \{ M_L : L \in \GLdR \}
    \end{align*}
    generates the symplectic group $\SptdR$, and $U_J, U_{V_P}$ and $U_{M_L}$ are unitary operators \related to these generators, i.e.\
    \begin{align*}
        U_J\rho(z) = \rho(Jz) U_J, \quad U_{V_P} \rho(z) = \rho(V_P z) U_{V_P} \quad \text{and} \quad U_{M_L} \rho(z) = \rho(M_L z) U_{M_L}
    \end{align*}
    for all $z \in \Rtd$. 
\end{prop}
\begin{proof}
    See Appendix \ref{appendix}.
\end{proof}

Finally, we note that \eqref{metaplectic_op} does \textit{not} hold if we replace the symmetrized time-frequency shifts $\rho(z)$ with ordinary time-frequency shifts $\pi(z)$. As mentioned in Section \ref{prelim_TF_Gabor}, the choice between $\rho(z)$ and $\pi(z)$ does have an impact on the analysis and synthesis operators associated to a Bessel sequence, so there are situations in Gabor analysis where using symmetrized time-frequency shifts is advantageous (for instance in the setting of modules over noncommutative tori). As long as one is concerned only with frame operators, as we will be for most of this article, the distinction rarely matters. 

\subsection{The Feichtinger Algebra}
\label{subsec_Falg}

The Feichtinger algebra $\Falg$ is a Banach space introduced by Feichtinger in 1981 \cite{Feichtinger_algebra}. It has many properties that makes it an ideal source of atoms for Gabor systems—see Jakobsen \cite{Jakobsen} for a thorough survey. The definition requires a bit of setup. 

For any $f, g \in L^2(\Rd)$, we define \textit{the short-time Fourier transform} of $f$ with \textit{window} $g$ to be the function $V_g f \from \Rtd \to \mathbb C$ defined by
\begin{align*}
    V_gf(z) = \langle f, \pi(z) g \rangle.
\end{align*}
If $g$ is $L^2$-normalized, the operator mapping a function $f \in L^2(\Rd)$ to the function $V_g f$ turns out to be an isometry from $L^2(\Rd)$ into $L^2(\Rtd)$. This really is the motivation behind the theory of Gabor frames, which can be seen as a discretized version of this construction (where we replace $\Rtd$ with a lattice $\Lambda$). There is one particularly nice window that we will make use of from time to time, namely the standard ($L^2$-normalized) Gaussian $g_0$ on $\Rd$,
\begin{align*}
    g_0(t) \coloneqq 2^{d/4} e^{-\pi |t|^2}.
\end{align*}

We now define the Feichtinger algebra as the space of all $f \in L^2(\Rd)$ whose short-time Fourier transforms with window $g_0$ lie in $L^1(\Rtd)$. 

\begin{define}[The Feichtinger Algebra]
    The set
    \begin{align*}
        \Falg \coloneqq \left\{ f \in L^2(\Rd) : \int_{\Rtd} | \langle f, \pi(z) g_0 \rangle | \, dz < \infty \right\}
    \end{align*}
    is called the \textit{Feichtinger algebra}. Fixing any nonzero $g \in \Falg$, the norm 
    \begin{align*}
        \| f \|_g \coloneqq \int_{\Rtd} | \langle f, \pi(z) g \rangle | \, dz
    \end{align*}
    turns $\Falg$ into a Banach space. Any other choice gives an equivalent norm.
\end{define}

\noindent The following lemma summarizes the properties we will need.

\begin{lemma}[Properties of the Feichtinger Algebra]
    \label{lemma_properties_of_Feichtinger}
    The following statements are true. 
    \begin{itemize}
        \item[(i)] For any $g \in \Falg$ and $A \in \GLtdR$, the Gabor system $\mathcal G(g, A)$ is a Bessel sequence.
        \item[(ii)] The Feichtinger algebra is invariant under all time-frequency shifts:\ if $g \in \Falg$, then $\pi(z) g \in \Falg$ for every $z \in \Rtd$.
        \item[(iii)] If $U$ is a unitary operator on $L^2(\Rd)$ that is \related to any symplectic matrix $S$, then $\Falg$ is invariant under $U$:\ if $g \in \Falg$, then $Ug \in \Falg$. (This is equivalent to saying that $\Falg$ is invariant under the metaplectic group.)
        \item[(iv)] For every nonzero $g \in \Falg$ and $A \in \GLtdR$, there exists $\epsilon > 0$ such that $\mathcal G(g, \epsilon A)$ is a frame and the canonical dual atom $(S^{\epsilon A}_{g, g})^{-1} g$ is in $\Falg$ as well. 
        \item[(v)] The Feichtinger algebra is dense in $L^2(\Rd)$. 
    \end{itemize}
\end{lemma}
\begin{proof}
    The proofs of all of these properties can be found in Gröchenig \cite[Chapters~11-13]{Grochenig_book}. In particular, see Theorem 13.1.1 in \cite{Grochenig_book} for point $(iv)$, which is by far the most delicate. We also remark that a much stronger version of the statement regarding the canonical dual atom is known to be true:\ if $g \in \Falg$ and $A \in \GLtdR$ are such that $\mathcal G(g, A)$ is a frame, then $S_{g, g}^A$ is a bijection on $\Falg$. In particular, the canonical dual atom is then in $\Falg$. This was proved by Gröchenig and Leinert \cite{Gröchenig_Leinert}. 
\end{proof}

We will also have need for the following exceedingly useful relation between frame operators over a lattice and its adjoint. It was discovered by Janssen \cite{Janssen} in the case of lattices of the form $\alpha\Zd \times \beta\Zd$ (with $\alpha, \beta > 0$) and atoms in the Schwartz class, and is therefore known as the \textit{Janssen representation}. It was then extended to more general lattices and atoms by Feichtinger and Zimmermann \cite{Feichtinger_Zimmermann_FIGA} and by the second named author in collaboration with Feichtinger \cite{Feichtinger_Luef_FIGA}. The generality we will need was obtained by Austad and Enstad \cite[Proposition~3.18]{Austad_Enstad}.

\begin{prop}[The Janssen Representation of the Frame Operator]
    \label{prop_Janssen}
    Let $A \in \GLtdR$ and $f, g, h \in L^2(\Rd)$ be such that $\mathcal G(f, A)$, $\mathcal G(g, A)$ and $\mathcal G(h, A)$ are Bessel sequences. Then, we have
    \begin{align*}
        S^A_{g, h} (f) = \frac{1}{|A|} S^{A^\circ}_{g, f}(h) = \frac{1}{|A|} \sum_{k \in \Ztd} \blangle h, \pi(A^\circ k) g \brangle \pi(A^\circ k) f,
    \end{align*}
    with (unconditional) convergence in $L^2(\Rd)$. By Theorem \ref{theorem_duality}, we may instead assume that $\mathcal G(f, A^\circ)$, $\mathcal G(g, A^\circ)$ and $\mathcal G(h, A^\circ)$ are Bessel sequences.
\end{prop}

\section{Classification of Lattices via Symplectic Forms}

\subsection{Lattices and Symplectic Forms}
\label{sec_lattices_and_forms}

With $J$ as in Equation \eqref{std_symp_matrix} and $A \in \GLtdR$, we can write the commutation relations \eqref{std_commutation_relation_with_J} on the lattice $A \Ztd$ as
\begin{align}
    \label{eq_final_commutation_relation}
    \pi(Ak)\pi(Al) = e^{2 \pi i l^T (A^T J A) k} \pi(Al) \pi(Ak) \quad \text{for all } k, l \in \Ztd.
\end{align}
We see that the commutation relations depend (only) on the symplectic form $A^T J A$ determined by $A$. This happens because we are parameterizing the lattice by $\Ztd$. We are simply doing a change of basis:\ the standard symplectic form $J$ over $\Lambda = A \Ztd$ becomes the symplectic form $A^T J A$ over $\Ztd$. This shows how the matrix $A^T J A$ is naturally linked to the lattice $A\Ztd$, for it controls the commutation relations of time-frequency shifts over that lattice! In Section \ref{sec_geometric_content}, we will discuss how $A^T J A$ relates to the geometry of the lattice.

Recall Lemma \ref{lemma_symplectic_conjugation}:\ given any $A, B \in \GLtdR$, we have $A^T J A = B^T J B$ if and only if $BA^{-1}$ is symplectic. Furthermore, if $BA^{-1}$ is symplectic, then certainly $A \Ztd$ and $B\Ztd = (BA^{-1}) A \Ztd$ are related by a symplectic transformation. This helps us understand why Theorem \ref{thm_covariance} (symplectic covariance) should be true:\ \textit{symplectic transformations preserve the commutation relations}. That is, for $S \in \SptdR$, we have
\begin{align*}
    \pi(SAk)\pi(SAl) = e^{2 \pi i l^T ((SA)^T J (SA)) k} \pi(SAl) \pi(SAk),
\end{align*}
and $(SA)^T J (SA) = A^T (S^T J S) A = A^T J A$.

There is a somewhat subtle point here. One could look at the commutation relations in the form of Equation \eqref{std_commutation_relation_with_J} with $z, w \in \Lambda$ and think that $J$ controls the commutation relations over $\Lambda$. The point, however, is that the phase factors that arise when commuting time-frequency shifts over $B \Ztd$ generally will differ from those that arise when commuting time-frequency shifts over $A \Ztd$, \textit{unless} $A^T J A = B^T J B$. This point becomes clearest when one thinks of the collection $\{ \pi(Ak) : k \in \Ztd \}$ as a projective representation of $\Ztd$ on $L^2(\Rd)$. But this is all just to motivate the appearance of the symplectic form $A^T J A$—our results will illustrate its importance.

We now extend Theorem \ref{thm_covariance} to encompass duality while formulating it in terms of symplectic forms. This formulation is much better suited to concrete examples and computations;\ one could hardly imagine a simpler condition to check than $A^T JA = B^T JB$. For the notion of $\rho$-relatedness, see Definition \ref{rho_related}.

\begin{theorem}
    \label{theorem_main} 
    Suppose $A, B \in \GLtdR$ are such that $A^T J A = B^T J B$. Choose any unitary operator $U$ on $L^2(\Rd)$ that is \related to the symplectic matrix $BA^{-1}$. Then, the following statements are true for all $g, h \in L^2(\Rd)$.
    \begin{itemize}
        \item[(i)] $\mathcal G(g, A)$ is a Bessel sequence if and only if $\mathcal G(Ug, B)$ is a Bessel sequence.
        \item[(ii)] $\mathcal G(g, A)$ is a Gabor frame if and only if $\mathcal G(Ug, B)$ is a Gabor frame.
        \item[(iii)] If $\mathcal G(g, A)$ and $\mathcal G(h, A)$ are Bessel sequences, then $S^B_{Ug, Uh} = U (S^A_{g, h}) U^*$.
    \end{itemize}
    In the case of $(i)$ (or $(ii)$), the optimal Bessel bounds (or frame bounds) of $\mathcal G(g, A)$ and $\mathcal G(Ug, B)$ are equal. Moreover, this correspondence respects duality:\ $(i)$, $(ii)$ and $(iii)$ are all true with $A^\circ$ and $B^\circ$ in place of $A$ and $B$, with the same operator $U$.
\end{theorem}
\begin{proof}
    The unitary operator $U$ exists by Proposition \ref{prop_metaplectic_ops}, and we have 
    \begin{align*}
        U \rho(Ak) = \rho((BA^{-1})A k) U = \rho(B k) U \quad \text{for all } k \in \Ztd.
    \end{align*}
    The resulting equality
    \begin{align*}
        \sum_{k \in \Ztd} \bl| \blangle f, \rho(Bk) Ug \brangle \br|^2 =  \sum_{k \in \Ztd} \bl| \blangle U^* f, \rho(Ak) g \brangle \br|^2 \quad \text{for all } f \in L^2(\Rd)
    \end{align*}
    easily gives $(i)$ and $(ii)$, as well as equality of the optimal bounds in both cases. For example, if we assume that $\mathcal G(g, A)$ is a Bessel sequence with Bessel bound $C > 0$, then 
    \begin{align*}
        \sum_{k \in \Ztd} \bl| \blangle f, \rho(Bk) Ug \brangle \br|^2 =  \sum_{k \in \Ztd} \bl| \blangle U^* f, \rho(Ak) g \brangle \br|^2 \leq C \| U^* f \|^2 = C\| f \|^2
    \end{align*}
    shows that $\mathcal G(Ug, B)$ is a Bessel sequence with Bessel bound $C$.

    For $(iii)$, we have
    \begin{align*}
        S^B_{Ug, Uh} (Uf) &= \sum_{k \in \Ztd} \blangle Uf, \rho(Bk) Ug \brangle \rho(Bk) Uh \\ &= \sum_{k \in \Ztd} \blangle f, \rho(Ak) g \brangle U \rho(Ak) h = U S^A_{g, h} (f)
    \end{align*}
    for all $f\in L^2(\Rd)$, so $S^B_{Ug, Uh} U = U S^A_{g, h}$. As for the claim regarding duality, we have
    \begin{align*}
        (B^\circ)(A^\circ)^{-1} = (-JB^{-T})(-J A^{-T})^{-1} = -J B^{-T} A^T J = -J (B A^{-1})^{-T} J.
    \end{align*}
    Since $BA^{-1}$ is symplectic, we have $J (BA^{-1}) = (BA^{-1})^{-T} J$, and hence
    \begin{align*}
        (B^\circ)(A^\circ)^{-1} = - JJ BA^{-1} = BA^{-1}.
    \end{align*}
    This means that any unitary operator $U$ that is \related to $BA^{-1}$ is also \related to $(B^\circ)(A^\circ)^{-1}$, so $U\rho(A^\circ k) = \rho(B^\circ k) U$, and the proofs of $(i), (ii)$ and $(iii)$ all carry through with $A^\circ$ and $B^\circ$ in place of $A$ and $B$.
\end{proof}

The conventional form of symplectic covariance (with the addition of duality) follows easily:

\begin{corollary}
\label{corollary_main}
Let $S \in \SptdR$ and choose any unitary operator $U$ on $L^2(\Rd)$ that is \related to $S$. Then, for every lattice $\Lambda \subset \Rtd$, the following holds.
\begin{itemize}
    \item[(i)] $\mathcal G(g, \Lambda)$ is a Bessel sequence if and only if $\mathcal G(Ug, S\Lambda)$ is a Bessel sequence.
    \item[(ii)] $\mathcal G(g, \Lambda)$ is a Gabor frame if and only if $\mathcal G(Ug, S\Lambda)$ is a Gabor frame.
    \item[(iii)] If $\mathcal G(g, \Lambda)$ and $\mathcal G(h, \Lambda)$ are Bessel sequences, then $S^{S\Lambda}_{Ug, Uh} = U (S^\Lambda_{g, h}) U^*$.
\end{itemize}
Moreover, as in Theorem \ref{theorem_main}, we have equality of the optimal Bessel bounds and frame bounds, and duality is respected:\ $(S\Lambda)^\circ = S\Lambda^\circ$.
\end{corollary}
\begin{proof}
    Choose a lattice matrix $A$ for $\Lambda$ and note that $B \coloneqq SA$ and $U$ satisfy the conditions of Theorem \ref{theorem_main}. For the equality $(S\Lambda)^\circ = S\Lambda^\circ$, note that we actually proved that $(B^\circ)(A^\circ)^{-1} = BA^{-1}$, and here $S = BA^{-1}$, so we have $S A^\circ = (B^\circ)(A^\circ)^{-1} A^\circ = B^\circ = (SA)^\circ$.
\end{proof}

Before we proceed to formalize the content of Theorem \ref{theorem_main}, we wish to give an example (due to de Gosson \cite{deGossonMauriceA.2015HdoG}) showing the power of Corollary \ref{corollary_main}. Given any $g \in L^2(\Rd)$, we denote by $\mathscr F(g)$ the \textit{frame set} of $g$. This is the set of those lattices over which $g$ generates a Gabor frame, i.e.\
\begin{align*}
    \mathscr F(g) \coloneqq \{ A\Ztd : \mathcal G(g, A) \text{ is a frame} \}.
\end{align*}
Consider an arbitrary $L^2$-normalized Gaussian on $\Rd$,
\begin{align}
    \label{eq_def_gaussian}
    g_{X, Y}(t) = 2^{d/4} (\det X)^{1/4}  e^{-\pi t^T(X + i Y)t},
\end{align}
where $X$ and $Y$ are real and symmetric $d \times d$ matrices and $X$ is positive definite. Let $X^{1/2}$ denote the positive definite square root of $X$, and consider the unitary operators $U_{V_Y}$ and $U_{M_{X^{1/2}}}$ on $L^2(\Rd)$ defined by
\begin{align*}
    U_{V_Y}f(t) = e^{-\pi i t^T Y t} f(t) \quad \text{and} \quad U_{M_{X^{1/2}}} f(t) = \sqrt{\det(X^{1/2})} f(X^{1/2}t).
\end{align*}
By Proposition \ref{prop_generators}, these are \related to the symplectic matrices
\begin{align*}
    V_Y = \begin{pmatrix}
        I & 0 \\ -Y & I
    \end{pmatrix} \quad \text{and} \quad M_{X^{1/2}} = \begin{pmatrix}
        (X^{1/2})^{-1} & 0 \\ 0 & (X^{1/2})^T
    \end{pmatrix} = \begin{pmatrix}
        X^{-1/2} & 0 \\ 0 & X^{1/2}
    \end{pmatrix},
\end{align*}
respectively. 

The following result is due to de Gosson \cite[Proposition~13]{deGossonMauriceA.2015HdoG}. We reproduce it here in more conventional Gabor-analytic terms in order to illustrate the power of symplectic covariance.

\begin{corollary}[Frame Sets of Gaussians; de Gosson]
    \label{corollary_frame_sets_Gaussians}
    Let $g_0(t) = 2^{d/4} e^{-\pi |t|^2}$ be the standard ($L^2$-normalized) Gaussian on $\Rd$, and let $\mathscr F(g_0)$ be its frame set. Then, the frame set of the Gaussian $g_{X, Y}$ (Equation \eqref{eq_def_gaussian}) is given by
    \begin{align*}
        \mathscr F(g_{X, Y}) = S(\mathscr F(g_0)) = \{ S\Lambda : \Lambda \in \mathscr F(g_0)  \}, \text{ where } S = \begin{pmatrix}
            X^{-1/2} & 0 \\ -YX^{-1/2} & X^{1/2}
        \end{pmatrix}.
    \end{align*}
\end{corollary}
\begin{proof}
    Since
    \begin{align*}
        g_{X, Y}(t) = 2^{d/4} (\det X)^{1/4} e^{-\pi t^T(X + i Y)t} = U_{V_Y} U_{M_{X^{1/2}}} g_0(t)
    \end{align*}
    and $U_{V_Y} U_{M_{X^{1/2}}}$ is \related to the symplectic matrix
    \begin{align*}
        S = V_Y M_{X^{1/2}} = \begin{pmatrix}
            I & 0 \\ -Y & I
        \end{pmatrix} \begin{pmatrix}
            X^{-1/2} & 0 \\ 0 & X^{1/2}
        \end{pmatrix} = \begin{pmatrix}
            X^{-1/2} & 0 \\ -YX^{-1/2} & X^{1/2}
        \end{pmatrix},
    \end{align*}
    Corollary \ref{corollary_main} tells us that $\mathcal G(g_0, \Lambda)$ is a frame if and only if $\mathcal G(g_{X, Y}, S\Lambda)$ is a frame. 
\end{proof}

\begin{remark}
    \label{remark_Franz_Xu}
    This result also follows from the correspondence between Gabor frames with Gaussian windows and sets of interpolation in Bargmann-Fock spaces (see Corollary 1.9 in the article \cite{Franz_Xu} by the second named author and Wang), but the perspective of symplectic covariance makes it much more explicit and also shows that it is an instance of a much more general phenomenon.
\end{remark}

We now introduce some terminology that will allow us to speak about Theorem \ref{theorem_main} in a more concise and precise manner.

\begin{define}[Symplectically Related Lattices]
\label{def_symplectically_related}
We say that two lattices $\Lambda$ and $\Lambda'$ in $\Rtd$ are \textit{symplectically related} if there exists some $S \in \SptdR$ such that $\Lambda' = S\Lambda$.
\end{define}

\begin{define}[Equivalence of Gabor Structures]
    \label{definition_main}
    Let $\Lambda$ and $\Lambda'$ be two lattices in $\Rtd$. We say that the \textit{Gabor structures}\footnote{The \textit{Gabor structure} over $\Lambda$ is supposed to be short for ``the structure of Gabor systems over $\Lambda$''.} over $\Lambda$ and $\Lambda'$ are \textit{equivalent} if there exists an invertible bounded linear operator $\F \from L^2(\Rd) \to L^2(\Rd)$ such that, for every $g, h \in L^2(\Rd)$,
    \begin{itemize}
            \item[(i)] $\mathcal G(\F g, \Lambda')$ is a Bessel sequence if and only if $\mathcal G(g, \Lambda)$ is a Bessel sequence, and
            \item[(ii)] $S^{\Lambda'}_{\F g, \F h} = \F (S^\Lambda_{g, h}) \F^{-1}$ whenever $\mathcal G(g, \Lambda)$ and $\mathcal G(h, \Lambda)$ are Bessel sequences.
    \end{itemize}
    We say that the equivalence is \textit{implemented} by $\F$. If $\F$ can be taken to be unitary, we say that the Gabor structures over $\Lambda$ and $\Lambda'$ are \textit{unitarily equivalent}.
\end{define}

\noindent It is a basic result in Gabor analysis that a Bessel sequence $\mathcal G(g, \Lambda)$ is a frame if and only if the frame operator $S_{g, g}^\Lambda$ on $L^2(\Rd)$ is invertible. This implies that if the Gabor structures over $\Lambda$ and $\Lambda'$ are equivalent, with the equivalence implemented by $\F$, then $\mathcal G(\F g, \Lambda')$ is a frame if and only if $\mathcal G(g, \Lambda)$ is a frame. Moreover, Bessel bounds and frame bounds can be calculated in terms of the spectra of the corresponding frame operators. Since the spectrum of an operator is invariant under conjugation by an invertible operator, the equality of optimal Bessel bounds and frame bounds also follows from equivalence of Gabor structures.\footnote{The invariance of the spectrum under conjugation follows immediately from its definition:\ $G-\alpha I$ is invertible if and only if $\F G\F^{-1} - \alpha I = \F(G - \alpha I)\F^{-1}$ is invertible.} Thus, the defining conditions for equivalence of Gabor structures give us all the properties we could want—they are chosen to be minimal. (This also shows that the conclusions in Theorem \ref{theorem_main} and Corollary \ref{corollary_main} are redundant:\ conclusion $(ii)$ and the remarks regarding optimal bounds follow from $(i)$ and $(iii)$.)

In the language of these definitions, Corollary \ref{corollary_main} says that if two lattices $\Lambda$ and $\Lambda'$ are symplectically related, then their structures of Gabor systems are unitarily equivalent, with the equivalence implemented by any unitary operator $U$ that is \related to the symplectic transformation taking $\Lambda$ to $\Lambda'$. Theorem \ref{theorem_main} gives us a very simple sufficient condition for checking whether two lattices are symplectically related:\ two lattices $\Lambda = A \Ztd$ and $\Lambda' = B\Ztd$ are symplectically related if $A^T J A = B^T J B$. Thus, if $A^T J A = B^T J B$, then the Gabor structures over $A \Ztd$ and $B \Ztd$ are unitarily equivalent. Since a skew symmetric $2d \times 2d$ matrix has $2d (2d-1)/2 = 2d^2 - d$ independent entries, this means that although there are $4d^2$ degrees of freedom in the choice of lattice matrix $A \in \GLtdR$, there are at most $2d^2 - d$ degrees of freedom when it comes to the structures of Gabor systems over those lattices. In Section \ref{sec_separable_lattices}, we will see that when it comes to separable lattices, i.e.\ lattices of the form $A_1 \Zd \times A_2 \Zd$ with $A_1, A_2 \in \GLdR$, there are at most $d^2$ degrees of freedom for Gabor structures, and they are in fact naturally parameterized by $\GLdR$ (not $\GLdR \times \GLdR$).

The interplay between lattice matrices and symplectic forms ($A \mapsto A^T J A$) is simple, but the interplay between \textit{lattices} and symplectic forms is complicated somewhat by the redundancy of lattice matrices described in Remark \ref{remark_redundancy_of_lattice_matrices}. Given a lattice $\Lambda = A \Ztd$, we can associate to it the \textit{set} 
\begin{align}
    \label{eq_set_of_symp_forms}
    \Theta_\Lambda \coloneqq \{ (A \M)^T J (A \M) :  \M \in \GLtdZ \}
\end{align}
of all symplectic forms determined by its set $A \GLtdZ$ of lattice matrices. These are easily seen to be either equal or disjoint:\footnote{In fact, $\Theta_\Lambda$ is the orbit of $A^T JA$ under the action of $\GLtdZ$ on $\GLtdR$ given by $T \mapsto M^T T M$, where $M \in \GLtdZ$ and $T \in \GLtdR$.} given any two lattices $\Lambda$ and $\Lambda'$, we have $\Theta_\Lambda = \Theta_{\Lambda'}$ if and only if $\Lambda' = S \Lambda$ for some $S \in \SptdR$ and $\Theta_\Lambda \cap \Theta_{\Lambda'} = \emptyset$ otherwise. In other words, $\Lambda$ and $\Lambda'$ are symplectically related \textit{if and only if} $\Theta_\Lambda \cap \Theta_{\Lambda'} \neq \emptyset$ (if and only if $\Theta_\Lambda = \Theta_{\Lambda'}$). 

Corollary \ref{corollary_main} can also be stated as follows:\ any symplectic transformation $S$ of the time-frequency plane $\Rtd$ lifts to a unitary operator $U$ on $L^2(\Rd)$ which, \textit{for any lattice} $\Lambda$, implements an equivalence between the Gabor structures over $\Lambda$ and $S\Lambda$. In Section \ref{sec_necessity}, we will prove that (modulo a minor complication related to complex conjugation) symplectic transformations are the only linear transformations of the time-frequency plane for which such operators exist (unitary or not). In essence, this means that—at the level of linear transformations of the time-frequency plane—the structure preserving transformations in Gabor analysis are precisely the symplectic ones.

\begin{remark}
\phantom{=}
\begin{itemize}
    \item In the $d=1$ case, we have $\mathrm{Sp}(2, \mathbb R) = \mathrm{SL}(2, \mathbb R)$, so two lattices in $\mathbb R^2$ are symplectically related if and only if they have the same covolume. For $d > 1$, $\SptdR$ is a proper subgroup of $\mathrm{SL}(2d, \mathbb R)$, so the story is more complicated.
    \item It is interesting to note that
    \begin{align*}
        (A^\circ)^T J A^\circ = (-JA^{-T})^T J (-JA^{-T}) = A^{-1} J A^{-T} = - (A^T J A)^{-1},
    \end{align*}
    so that—in the context of symplectic forms—the duality in Gabor analysis corresponds to the involution $\theta \mapsto -\theta^{-1}$.
\end{itemize}
\end{remark}

\begin{remark}
    Balan \cite{Balan} has defined a notion of unitary equivalence of frames in the context of general separable Hilbert spaces. He also defines a metric on each equivalence class, thus quantifying the distance between two equivalent frames. Our notion of equivalence (which pertains to lattices) is distinct from his in many respects. There is, however, the following connection. Consider any symplectic transformation $S$, any \related unitary operator $U$ on $L^2(\Rd)$ and an arbitrary lattice matrix $A \in \GLtdR$. The fact that $U$ and $S$ satisfy \eqref{metaplectic_op} implies that (\textit{if} we use symmetrized time-frequency shifts) the Gabor systems $\mathcal G(g, A)$ and $\mathcal G(Ug, SA)$ are $U$-equivalent in the sense of Balan whenever $g$ generates a frame over $A\Ztd$. Moreover, by Theorem 2.4 in \cite{Balan}, the distance between them is given by 
    \begin{align*}
        d\bl(\mathcal G(g, A), \mathcal G(Ug, SA) \br) = \log \bl( 1 + \max\bl( \|U - I\|, \|I - U^* \|\br) \br) = \log(1 + \| U-I \|).
    \end{align*}   
    This depends only on the operator $U$, so we actually have a function
    \begin{align*}
        \phi \from \MpctdR &\to [0, \infty) \\
        U &\mapsto \log(1 + \| U-I \|)
    \end{align*}
    with the property that 
    \begin{align*}
         d\bl(\mathcal G(g, A), \mathcal G(Ug, SA) \br) = \phi(U)
    \end{align*}
    for every $A \in \GLtdR$ and any $g \in L^2(\Rd)$ generating a frame over $A \Ztd$. Thus, the distance between Gabor frames related in this manner is independent of both the lattice and the atom.
\end{remark}

\subsection{Separable Lattices}
\label{sec_separable_lattices}

We call a lattice $\Lambda \subset \Rtd$ \textit{separable} if it is of the form
\begin{align*}
    \Lambda = \begin{pmatrix}
        A_1 & 0 \\ 0 & A_2
    \end{pmatrix} \Ztd \quad \text{for some } A_1, A_2 \in \GLdR.
\end{align*}
For a separable lattice, the translation and modulation parameters are separate, in the sense that $\pi(A (k_1, k_2)) = M_{A_2 k_2} T_{A_1 k_1}$ for all $k_1, k_2 \in \Zd$. This greatly simplifies the structure of Gabor systems, so there are many results that apply only to separable lattices (see e.g. Ron and Shen \cite{RonShen}). The following proposition characterizes the symplectic forms determined by lattices that are symplectically related to separable ones. In Section \ref{subsec_separable_geometric}, we will give a more symplecto-geometric characterization of such lattices.

\begin{prop}[The Symplectic Forms of Separable Lattices]
    \label{prop_separable_latties}
    Let $\Lambda \subset \Rtd$ be a lattice. Then, $\Lambda$ is symplectically related to a separable lattice if and only if there exists $K \in \GLdR$ such that $\left(\begin{smallmatrix} 0 & K \\ - K^T & 0 \end{smallmatrix}\right) \in \Theta_\Lambda$. In other words, the following conditions are equivalent.
    \begin{itemize}
        \item[(a)] There exists $A \in \GLtdR$ such that $\Lambda = A\Ztd$ and $A^T J A = \left(\begin{smallmatrix} 0 & K \\ - K^T & 0 \end{smallmatrix}\right)$ for some $K \in \GLdR$.
        \item[(b)] There exists $B \in \GLtdR$ of the form
        \begin{align*}
            B = \begin{pmatrix}
                B_1 & 0 \\ 0 & B_2
            \end{pmatrix} \quad \text{with } B_1, B_2 \in \GLdR
        \end{align*}
        and $S \in \SptdR$ such that $\Lambda = SB \Ztd$.
    \end{itemize}
\end{prop}
\begin{proof}
    If $(b)$ holds, then $(a)$ holds with $A = SB$ and $K = B_1^T B_2$, since
    \begin{align*}
        (SB)^T J (SB) = B^T J B = \begin{pmatrix}
            0 & B_1^T B_2 \\ - B_2^T B_1 & 0
        \end{pmatrix}.
    \end{align*} 
    Suppose now that $(a)$ holds. The calculation
    \begin{align*}
        \begin{pmatrix}
            I & 0 \\ 0 & K
        \end{pmatrix}^T  \begin{pmatrix}
            0 & I \\ -I & 0
        \end{pmatrix} \begin{pmatrix}
            I & 0 \\ 0 & K
        \end{pmatrix} = \begin{pmatrix}
            0 & K \\ -K^T & 0
        \end{pmatrix}
    \end{align*}
    shows that $B \coloneqq \left(\begin{smallmatrix} I & 0 \\ 0 & K \end{smallmatrix}\right)$ determines the same symplectic form as $A$. By Lemma \ref{lemma_symplectic_conjugation}, $S \coloneqq AB^{-1}$ is symplectic, which gives $(b)$.
\end{proof}

In light of the proof, we obtain the following corollary.

\begin{corollary}
    \label{corollary_separable_lattices}
    Let $A \in \GLtdR$ and suppose that 
    \begin{align*}
        A^T J A = \begin{pmatrix}
            0 & K \\ -K^T & 0
        \end{pmatrix} \quad \text{for some } K \in \GLdR.
    \end{align*}
    Then, the lattice $A\Ztd$ is symplectically related to the separable lattice
    \begin{align*}
        \begin{pmatrix}
            I & 0 \\ 0 & K
        \end{pmatrix} \Ztd,
    \end{align*}
    so by Theorem \ref{theorem_main}, the structure of Gabor systems over $A \Ztd$ is unitarily equivalent to the structure of Gabor systems over $\left(\begin{smallmatrix} I & 0 \\ 0 & K \end{smallmatrix}\right) \Ztd$.
\end{corollary}
\begin{proof}
    Immediate from the proof of Proposition \ref{prop_separable_latties}.
\end{proof}

\begin{remark}[Finding the Unitary Operator]
    \label{remark_finding_the_unitary}
    In the context of Corollary \ref{corollary_separable_lattices}, if we decompose $A$ into $d \times d$ blocks, so that
    \begin{align*}
        A = \begin{pmatrix}
            A_{11} & A_{12} \\ A_{21} & A_{22}
        \end{pmatrix}, \quad \text{then } S \coloneqq \begin{pmatrix}
            A_{11} & A_{12}K^{-1} \\ A_{21} & A_{22}K^{-1}
        \end{pmatrix}
    \end{align*}
    is the symplectic matrix satisfying  $S \left(\begin{smallmatrix} I & 0 \\ 0 & K \end{smallmatrix}\right) = A$. If $A_{12}$ is invertible, then Proposition \ref{prop_free_metaplectic_ops} gives us the explicit form of a unitary operator $U$ that implements the equivalence of Gabor structures over $A\Ztd$ and $\left(\begin{smallmatrix} I & 0 \\ 0 & K \end{smallmatrix}\right) \Ztd$.
\end{remark}

Taken together, Proposition \ref{prop_separable_latties} and Corollary \ref{corollary_separable_lattices} show that—in terms of equivalence classes of symplectically related lattices—the equivalence class of any separable lattice has a representative of the form $\left(\begin{smallmatrix} I & 0 \\ 0 & K \end{smallmatrix}\right) \Ztd$ with $K \in \GLdR$. Thus, the $d^2$ parameters of $\GLdR$ suffice to parameterize the space of separable lattices in $\Rtd$ when it comes to the possible structures of Gabor systems:\ we capture all possible structures of Gabor systems (up to unitary equivalence) by considering lattices of the form $\left(\begin{smallmatrix} I & 0 \\ 0 & K \end{smallmatrix}\right) \Ztd$. 

We can also explain the reduction from $\GLdR \times \GLdR$ (the space of all separable lattice matrices) to $\GLdR$ in the following manner. The separable lattice $\left(\begin{smallmatrix} A_1 & 0 \\ 0 & A_2 \end{smallmatrix}\right) \Ztd$ is symplectically related to the separable lattice
\begin{align*}
    \begin{pmatrix}
        L^{-1} A_1 & 0 \\ 0 & L^T A_2
    \end{pmatrix}\Ztd, \quad \text{for any } L \in \GLdR,
\end{align*}
because the matrix $\left(\begin{smallmatrix} L^{-1} & 0 \\ 0 & L^T \end{smallmatrix}\right)$ is symplectic (see the discussion prior to Proposition \ref{prop_generators}). In particular, taking $L = A_1$, every separable lattice is symplectically related to one of the form $\left(\begin{smallmatrix} I & 0 \\ 0 & K \end{smallmatrix}\right) \Ztd$ with $K \in \GLdR$. We also see that $K = A_1^TA_2$, so that $A_1^T A_2$ is the (only) quantity that is relevant to the structure of Gabor systems over $\left(\begin{smallmatrix} A_1 & 0 \\ 0 & A_2 \end{smallmatrix}\right) \Ztd$.

We are now ready to justify the first example from the introduction.

\begin{example}
    \label{example_main}
    Consider the $d=2$ case and the lattice matrix
    \begin{align*}
        A = \begin{pmatrix}
            1 & \pi^2 & \frac{\pi}3 & \pi
            \\[\xs em] \pi^2 & 1 & \pi & \frac{\pi}2 
            \\[\xs em] 0 & \pi & \frac13 & 1
            \\[\xs em] \pi & 0 & 1 & \frac12
        \end{pmatrix}.
    \end{align*}
    We find that
    \begin{align*}
        A^T J A = \begin{pmatrix}
            0 & 0 & \frac13 & 1
            \\[\xs em] 0 & 0 & 1 & \frac12 
            \\[\xs em] -\frac13 & -1 & 0 & 0
            \\[\xs em] -1 & -\frac12 & 0 & 0
        \end{pmatrix} = \begin{pmatrix}
            0 & K \\ -K^T & 0
        \end{pmatrix}, \quad \text{with } K \coloneqq \begin{pmatrix}
            \frac13 & 1 \\ 1 & \frac12
        \end{pmatrix}.
    \end{align*}
    By Corollary \ref{corollary_separable_lattices}, this means that $A \mathbb Z^4$ is symplectically related to the separable (and rational) lattice
    \begin{align*}
        B \mathbb Z^4 = \begin{pmatrix}
            1 & 0 & 0 & 0 
            \\[\xs em] 0 & 1 & 0 & 0 
            \\[\xs em] 0 & 0 & \frac13 & 1 
            \\[\xs em] 0 & 0 & 1 & \frac12
        \end{pmatrix} \mathbb Z^4,
    \end{align*}
    with the symplectic transformation $S$ for which $A\mathbb Z^4 = S(B\mathbb Z^4)$ given by 
    \begin{align*}
        S = AB^{-1} = \begin{pmatrix}
            1 & \pi^2 & \frac{\pi}3 & \pi
            \\[\xs em] \pi^2 & 1 & \pi & \frac{\pi}2 
            \\[\xs em] 0 & \pi & \frac13 & 1
            \\[\xs em] \pi & 0 & 1 & \frac12
        \end{pmatrix} \begin{pmatrix}
            1 & 0 & 0 & 0 
            \\[\xs em] 0 & 1 & 0 & 0 
            \\[\xs em] 0 & 0 & \frac13 & 1 
            \\[\xs em] 0 & 0 & 1 & \frac12
        \end{pmatrix}^{-1} = \begin{pmatrix}
            1 & \pi^2 & \pi & 0
            \\[\xs em] \pi^2 & 1 & 0 &\pi 
            \\[\xs em] 0 & \pi & 1 & 0 
            \\[\xs em] \pi & 0 & 0 & 1
        \end{pmatrix}.
    \end{align*}
    By Theorem \ref{theorem_main}, this means that the structure of Gabor systems over $A \mathbb Z^4$ is unitarily equivalent to the structure of Gabor systems over the separable and rational lattice $B \mathbb Z^4$. Moreover, since $\left(\begin{smallmatrix} \pi & 0 \\ 0 & \pi \end{smallmatrix}\right)$ is invertible, Proposition \ref{prop_free_metaplectic_ops} tells us that this equivalence is implemented by the unitary operator $U$ defined by
    \begin{align*}
        Uf(t) =  \frac{1}{\pi} \int_{\mathbb R^2} f(x) e^{i(t \cdot t - 2 x \cdot t + x \cdot Mx)} \,dx, \quad \text{where } M = \begin{pmatrix}
            1 & \pi^2 \\ \pi^2 & 1
        \end{pmatrix},
    \end{align*}
    for $f \in L^2(\mathbb R^2)$ and $t \in \mathbb R^2$.
\end{example}

\subsection{The Geometric Content of Symplectic Forms}
\label{sec_geometric_content}

We can write the standard symplectic form on $\Rtd$ as the differential two-form
\begin{align*}
    \sigma = \sum_{j=1}^d d\omega_j \wedge dx_j,
\end{align*}
where $(x_1, \ldots, x_d, \omega_1, \ldots, \omega_d)$ are the standard coordinates on $\Rtd$, as usual. This equality is easily verified by applying each side to the standard basis vectors. By the usual geometric interpretation of differential two-forms on a Euclidean space:\ given $z, w \in \Rtd$, the number $d \omega_j \wedge d x_j(z, w)$ is obtained by orthogonally projecting the parallelogram spanned by $z$ and $w$ onto the $(x_j, \omega_j)$-plane and taking the area of this projection (with an orientation-dependent sign). Thus, $\sigma(z, w)$ is obtained by taking the parallelogram spanned by $z$ and $w$, projecting this onto each \textit{canonical plane} $(x_j, \omega_j)$ and summing up the areas of these projections (with signs depending on orientations).

Suppose now that we are given a lattice $A \Ztd$. Let us denote the columns of $A$ by $A_1, \ldots, A_{2d}$. The $(i, j)$-th entry of the symplectic form $A^T J A$ is precisely
\begin{align*}
    \sigma(A_j, A_i) = (A_i)^T J A_j = (A^T J A)_{ij}.
\end{align*}
Thus, by Theorem \ref{theorem_main}, the structure of Gabor systems over $A \Ztd$ is uniquely determined by the projections of the parallelograms spanned by $A_i$ and $A_j$ (as $i$ and $j$ range from $1$ to $2d$) onto the canonical planes $(x_j, \omega_j)$ (as $j$ ranges from $1$ to $d$). 

This observation ties in neatly with Bourouihiya's higher-dimensional version \cite[Corollary 4.3]{Bourouihiya} of the famous result of Lyubarskii, Seip and Wallstén \cite{Lyubarskii, WallsténSeip} on Gaussian Gabor frames. Note that combining this result with Corollary \ref{corollary_frame_sets_Gaussians} gives us frames generated by any given Gaussian.

\begin{theorem}[Bourouihiya]
    \label{Gaussian_Gabor_frames}
    Let $\alpha_j, \beta_j \in \mathbb R \backslash \{0\}$ for $1 \leq j \leq d$. Then, the standard Gaussian $g_0(t) = 2^{d/4} e^{-\pi |t|^2}$ generates a Gabor frame over the lattice
    \begin{align*}
        \Lambda = \biggl( \prod_{j=1}^d \alpha_j \mathbb Z  \biggr) \times \biggl( \prod_{j=1}^d \beta_j \mathbb Z  \biggr)
    \end{align*}
    if and only if $|\alpha_j \beta_j| < 1$ for $1 \leq j \leq d$.
\end{theorem}

\noindent Now, we can choose the $\alpha$'s and $\beta$'s in such a manner that the covolume of $\Lambda$ is arbitrarily small while keeping e.g.\ $\alpha_1 \beta_1 > 1$, so this result shows that \textit{covolume is not everything} in multivariate Gabor analysis (a similar remark has been made by Heil \cite{Heil2008}). However, the products $\alpha_j \beta_j$ are precisely the entries of the symplectic form determined by the obvious lattice matrix for $\Lambda$. In the $d = 2$ case, for example:
\begin{align*}
    A = \begin{pmatrix}
        \alpha_1 & 0 & 0 & 0 \\ 0 & \alpha_2 & 0 & 0 \\ 0 & 0 & \beta_1 & 0 \\ 0 & 0 & 0 & \beta_2
    \end{pmatrix} \; \implies \; A^T J A = \begin{pmatrix}
        0 & 0 & \alpha_1 \beta_1 & 0 \\ 0 & 0 & 0 & \alpha_2 \beta_2 \\ - \alpha_1 \beta_1 & 0 & 0 & 0 \\ 0 & -\alpha_2\beta_2 & 0 & 0
    \end{pmatrix}.
\end{align*}
Thus, the condition in Theorem \ref{Gaussian_Gabor_frames} is a condition on the entries of the symplectic form $A^T J A$ associated to $A\Ztd$. We will generalize this result in the upcoming section (see Theorem \ref{Gaussian_Gabor_frames_v2}). %For a discussion of how the entries of $A^T JA$ relate to the covolume of $A\Ztd$, see Section \ref{subsec_covolume}.

Finally, we note how the covolume of an arbitrary lattice $\Lambda = A \Ztd$ relates to the entries of $A^T J A$. Of course, we always have that $|A|^2 = \det (A^T J A)$, just by the properties of the determinant and the fact that $\det J = 1$, but there is a significantly more subtle relationship. It has long been known that the determinant of any antisymmetric matrix $\theta$ can be written as the square of a polynomial in its entries \cite{Ledermann}. This polynomial (as well as its value) is known as the \textit{Pfaffian} of $\theta$ and is denoted by $\pf \theta$. In the $d = 2$ case, for example, we have
\begin{align*}
    \pf \begin{pmatrix}
        0 & a & b & c \\ -a & 0 & d & e \\ -b & -d & 0 & f \\ -c & -e & -f & 0 
    \end{pmatrix} = a f - b e + c d.
\end{align*}
The Pfaffian has the property that for any $A \in \MtdR$ (and any antisymmetric $\theta \in \MtdR$), we have $\pf (A^T \theta A) = \det(A) \pf(\theta)$. Taking $J = \theta$, we therefore get the relation $|A| = \pf (A^TJA)$. Thus, the covolume of the lattice $A\Ztd$ is a very natural and well-studied invariant of the antisymmetric matrix $A^T J A$, and in fact a polynomial in its entries.

\section{A Result on Multivariate Gaussian Gabor Frames}
\label{sec_Gaussians}

In this section, we use our framework to generalize Theorem \ref{Gaussian_Gabor_frames}, the standard higher-dimensional version of the Lyubarskii-Seip-Wallstén theorem on Gaussian Gabor frames. The result is Theorem \ref{Gaussian_Gabor_frames_v2}.

Recall that an arbitrary Gaussian takes the form
\begin{align*}
    g_{X, Y}(t) = U_{V_Y} U_{M_{X^{1/2}}} g_0(t) = C e^{-\pi t^T(X + i Y)t},
\end{align*}
where $C$ is a constant typically chosen so that $g_{X, Y}$ is $L^2$-normalized, $X$ and $Y$ are real and symmetric, $X$ is positive definite, and $U_{V_Y}$ and $U_{M_{X^{1/2}}}$ are given by \eqref{def_of_U_V_and_U_M}. The choice of constant $C \neq 0$ does not matter when it comes to the question of whether $g_{X, Y}$ generates a Gabor frame, so we will leave it arbitrary in this section.

We begin by noting that symplectic covariance gives us some freedom of choice for the Gaussian in Theorem \ref{Gaussian_Gabor_frames}. We will incorporate this freedom in Theorem \ref{Gaussian_Gabor_frames_v2}, where it essentially counteracts the arbitrary choice we make when using $\Zd \times K \Zd$ as a representative for the equivalence class of lattices determined by the symplectic form $\left(\begin{smallmatrix} 0 & K \\ -K & 0 \end{smallmatrix}\right).$

\begin{remark}
    Let $D \in \GLdR$ be diagonal. Then, the Gaussian
    \begin{align*}
        g_{D^2, 0}(t) = C e^{-\pi t^T D^2 t} \quad (C \neq 0)
    \end{align*}
    generates a Gabor frame over the lattice
    \begin{align*}
        \Lambda = \biggl( \prod_{j=1}^d \alpha_j \mathbb Z  \biggr) \times \biggl( \prod_{j=1}^d \beta_j \mathbb Z  \biggr) \quad (\alpha_j, \beta_j \neq 0)
    \end{align*}
    if and only if $|\alpha_j \beta_j| < 1$ for $1 \leq j \leq d$.
\end{remark}
\begin{proof}
    Write $D = \mathrm{diag}(\gamma_1, \ldots, \gamma_d)$ and consider the symplectic matrix
    \begin{align*}
        M_{D^{-1}} = \begin{pmatrix}
            D & 0 \\ 0 & D^{-1}
        \end{pmatrix} = \mathrm{diag}(\gamma_1, \ldots, \gamma_d, \gamma_1^{-1}, \ldots, \gamma_d^{-1}).
    \end{align*}
    By Theorem \ref{Gaussian_Gabor_frames}, the standard Gaussian, $g_0$, generates a Gabor frame over the lattice
    \begin{align*}
        M_{D^{-1}} \Lambda = \biggl( \prod_{j=1}^d \gamma_j \alpha_j \mathbb Z  \biggr) \times \biggl( \prod_{j=1}^d \gamma_j^{-1} \beta_j \mathbb Z  \biggr)
    \end{align*}
    if and only if $|\alpha_j \beta_j| = |(\gamma_j \alpha_j)(\gamma_j^{-1} \beta_j)| < 1$ for all $j$. By Corollary \ref{corollary_main}, this is equivalent to the Gaussian $g_{D^2, 0} = U_{M_D}g_0$ generating a Gabor frame over the lattice $\Lambda = M_D M_{D^{-1}} \Lambda$.
\end{proof}

Now for the main result. Note that it reduces to the preceding remark in the special case of $A = \mathrm{diag}(\alpha_1, \ldots, \alpha_d, \beta_1, \ldots, \beta_d)$. The proof relies on the \textit{pre-Iwasawa factorization} of a symplectic matrix, which is stated and proved in Appendix \ref{sec_pre_Iwasawa}.

\begin{theorem}[Gaussian Gabor Frames]
    \label{Gaussian_Gabor_frames_v2}
    Suppose $A \in \GLtdR$ is such that 
    \begin{align*}
        A^T J A = \begin{pmatrix}
            0 & K \\ -K & 0
        \end{pmatrix} \quad \text{with } K \text{ diagonal.}
    \end{align*} 
    Write $A = \left(\begin{smallmatrix} A_{11} & A_{12} \\ A_{21} & A_{22} \end{smallmatrix}\right)$ where the $A_{ij}$'s are $d \times d$ blocks. Choose any diagonal matrix $D \in \GLdR$ and let
    \begin{align*}
        X &= (A_{11} D^2 A_{11}^T + A_{12} (DK)^{-2} A_{12}^T)^{-1} \\ \text{and} \quad Y &= -(A_{21} D^2 A_{11}^T + A_{22} (DK)^{-2} A_{12}^T) X
    \end{align*}
    ($X$ is always well-defined). Then, the Gaussian
    \begin{align*}
        g_{X, Y}(t) = C e^{-\pi t^T(X + i Y)t} \quad (C \neq 0)
    \end{align*}
    generates a Gabor frame over $A \Ztd$ if and only if every entry of $K$ has absolute value $< 1$.
\end{theorem}

\noindent Note that $A^T JA$ is antisymmetric for any $A \in \GLtdR$, so that if the upper half of $A^TJA$ takes the form $\begin{pmatrix}
    0 & K
\end{pmatrix}$ with $K$ diagonal, then $A$ satisfies the condition of Theorem \ref{Gaussian_Gabor_frames_v2}.

\begin{remark}[Symplectic Lattices]
    Gröchenig \cite{Grochenig_book} defines a \textit{symplectic lattice} to be any lattice of the form $\alpha S \Ztd$ with $\alpha \neq 0$ and $S \in \SptdR$. All such lattices satisfy the condition of Theorem \ref{Gaussian_Gabor_frames_v2} with $K = \alpha^2 I$.
\end{remark}

Before we prove Theorem \ref{Gaussian_Gabor_frames_v2}, we illustrate it with a couple of examples.

\begin{example}
    \label{Gaussian_example}
    Consider the lattice matrix 
    \begin{align*}
        A = \begin{pmatrix}
            3 & 0 & \frac12 & 0
            \\[\xs em] 0 & 1 & 0 & 1 
            \\[\xs em] 14 & 1 & \frac52 & 1
            \\[\xs em] 3 & \frac{14}3 & \frac12 & 5
        \end{pmatrix}.
    \end{align*}
    We find that
    \begin{align*}
        A^T J A = \begin{pmatrix}
            0 & 0 & \frac12 & 0
            \\[\xs em] 0 & 0 &  0 & \frac13 
            \\[\xs em] -\frac12 & 0 & 0 & 0
            \\[\xs em] 0 & -\frac13 & 0 & 0
        \end{pmatrix} = \begin{pmatrix}
            0 & K \\ -K & 0
        \end{pmatrix}, \quad \text{with } K \coloneqq \begin{pmatrix}
            \frac12 & 0 \\ 0 & \frac13
        \end{pmatrix}.
    \end{align*}
    Theorem \ref{Gaussian_Gabor_frames_v2} now gives us a two-parameter family of Gaussians generating Gabor frames over $A \mathbb Z^4$. Indeed, let $D= \left( \begin{smallmatrix} a & 0 \\ 0 & b \end{smallmatrix}\right)$ with $a, b \in \mathbb R$ nonzero. With the notation of Theorem \ref{Gaussian_Gabor_frames_v2}, we find that
    \begin{align*}
        X = \begin{pmatrix}
            \frac{a^2}{9a^4+1} & 0 
            \\[\xs em] 0 & \frac{b^2}{b^4 + 9}
        \end{pmatrix} \quad \text{and} \quad Y = - \begin{pmatrix}
            \frac{42a^4 + 5}{9a^4+1} & 1 
            \\[\xs em] 1 & \frac{14 b^4 + 135}{3(b^4 + 9)}.
        \end{pmatrix}
    \end{align*}
    Thus, for all (nonzero) $a$ and $b$,
    \begin{align*}
        g(t) = \exp \left( - \pi t^T \begin{pmatrix}
            \frac{a^2 - i (42a^4 + 5) }{9a^4+1} & -i 
            \\[\xs em] -i & \frac{3b^2 - i (14 b^4 + 135) }{3(b^4 + 9)}
        \end{pmatrix} t \right)
    \end{align*}
    generates a Gabor frame over $A\mathbb Z^4$.
\end{example}

\begin{example}
    \label{Gaussian_example_2}
    Consider the lattices $A\mathbb Z^4$ and $B\mathbb Z^4$, where 
    \begin{align*}
        A = \begin{pmatrix}
            -\frac{14}{3} & -\frac{4}{21} & \frac{1}{25} & \frac{2}{5} 
            \\[\xs em] \frac{16}{3} & \frac{1}{21} & -\frac{2}{25} & \frac{1}{5} 
            \\[\xs em] -\frac{14}{3} & -\frac{11}{24} & \frac{1}{100} & \frac{19}{40} 
            \\[\xs em] \frac{91}{12} & -\frac{1}{6} & -\frac{13}{200} & \frac{7}{20}
        \end{pmatrix} \quad \text{and} \quad B = \begin{pmatrix}
            2\pi & -\frac{17}{9} & -\frac{2}{65} & \frac{3\sqrt{2}}{13} 
            \\[\xs em] \frac{\pi}{3} & -2 & \frac{3}{65} & \frac{2\sqrt{2}}{13}
            \\[\xs em] \frac{63\pi}{8} & -\frac{47}{12} & -\frac{57}{520} & \frac{33}{26\sqrt{2}} 
            \\[\xs em] -\frac{11\pi}{4} & -\frac{37}{8} & \frac{3}{20} & \frac{3}{4\sqrt{2}}
        \end{pmatrix}.
    \end{align*}
    We find that
    \begin{align*}
        A^T J A = \begin{pmatrix}
            0 & 0 & \frac25 & 0
            \\[\xs em] 0 & 0 &  0 & \frac17 
            \\[\xs em] -\frac25 & 0 & 0 & 0
            \\[\xs em] 0 & -\frac17 & 0 & 0
        \end{pmatrix} \quad \text{and}  \quad B^T J B = \begin{pmatrix}
            0 & 0 & \frac{\pi}5 & 0
            \\[\xs em] 0 & 0 &  0 & -\frac{\sqrt{2}}{3} 
            \\[\xs em] -\frac{\pi}5 & 0 & 0 & 0
            \\[\xs em] 0 & \frac{\sqrt{2}}{3} & 0 & 0
        \end{pmatrix}.
    \end{align*}
    Since $\tfrac{2}{5}, \tfrac{1}{7}, \tfrac{\pi}{5}, - \tfrac{\sqrt{2}}{3} \in (-1, 1)$, Theorem \ref{Gaussian_Gabor_frames_v2} gives us a two-parameter family of Gaussian Gabor frames for each of these lattices. 
\end{example}

\begin{proof}[Proof of Theorem \ref{Gaussian_Gabor_frames_v2}]
    Since 
    \begin{align*}
        A^T J A = \begin{pmatrix}
            0 & K \\ -K & 0
        \end{pmatrix} = \begin{pmatrix}
            I & 0 \\ 0 & K
        \end{pmatrix}^T J \begin{pmatrix}
            I & 0 \\ 0 & K
        \end{pmatrix},
    \end{align*}
    the matrix $A \left(\begin{smallmatrix} I & 0 \\ 0 & K \end{smallmatrix}\right)^{-1}$ is symplectic, and so
    \begin{align*}
        S \coloneqq A \begin{pmatrix}
            I & 0 \\ 0 & K
        \end{pmatrix}^{-1} M_{D^{-1}} = \begin{pmatrix}
            A_{11} D & A_{12} (DK)^{-1} \\ A_{21} D & A_{22} (DK)^{-1}
        \end{pmatrix}
    \end{align*}
    is symplectic as well. Applying the pre-Iwasawa factorization (Proposition \ref{prop_pre_Iwasawa}) to $S$, we have the following factorization of $S$ into symplectic matrices:
    \begin{align}
        \label{pre_Iwasawa}
        S = \begin{pmatrix}
            I & 0 \\ -Y & I
        \end{pmatrix} \begin{pmatrix}
            X^{-1/2} & 0 \\ 0 & X^{1/2}
        \end{pmatrix} \begin{pmatrix}
            P & Q \\ -Q & P
        \end{pmatrix} = V_Y M_{X^{1/2}} O,
    \end{align}
    where $O \coloneqq \left(\begin{smallmatrix} P & Q \\ -Q & P \end{smallmatrix}\right)$,
    \begin{align*}
        X &= (A_{11} D^2 A_{11}^T + A_{12}(DK)^{-2} A_{12}^T)^{-1},
        \\ Y &= - (A_{21}D^2A_{11}^T + A_{22}(DK)^{-2} A_{12}^T)X,
        \\ P &= X^{1/2} A_{11} D \quad \text{and} \quad Q = X^{1/2} A_{12} (DK)^{-1}.
    \end{align*}
    For the well-definedness of $X$, see the first paragraph of the proof of Proposition \ref{prop_pre_Iwasawa}.
    
    We now claim that every unitary operator $U_O$ that is $\rho$-related to $O$ fixes the standard Gaussian, up to a phase factor $\alpha$, i.e.\ $U_O g_0 = \alpha g_0$ with $|\alpha| = 1$. Assuming this for the time being, we can choose $U_O$ so that $U_O g_0 = g_0$. We know that $U_{V_Y} U_{M_{X^{1/2}}} U_O$ is $\rho$-related to $S$, so by Corollary \ref{corollary_main}, the Gaussian 
    \begin{align*}
        g_{X, Y} = U_{V_Y} U_{M_{X^{1/2}}} g_0 = U_{V_Y} U_{M_{X^{1/2}}} U_O g_0
    \end{align*}
    generates a Gabor frame over the lattice
    \begin{align*}
        S(D^{-1} \Zd \times DK\Zd) = SM_D\begin{pmatrix}
            I & 0 \\ 0 & K
        \end{pmatrix} (\Zd \times \Zd) = A \Ztd
    \end{align*}
    if and only if $g_0$ generates a Gabor frame over $D^{-1} \Zd \times (DK) \Zd$, which happens if and only if every entry of $K = D^{-1}(DK)$ has absolute value $< 1$ (by Theorem \ref{Gaussian_Gabor_frames}).

    All that remains is to justify that $U_O g_0 = \alpha g_0$ whenever $U_O$ is $\rho$-related to $O$. We will follow an argument of de Gosson \cite[Proposition~252]{deGosson_book}. It is well-known (and follows from a straightforward calculation) that
    \begin{align*}
        \langle g_0, \rho(z) g_0 \rangle = e^{\pi i x \cdot \omega} V_{g_0} g_0(z) = e^{- \frac \pi 2 z^T z} \quad \text{for all } z = (x, \omega) \in \Rtd.
    \end{align*}
    Since $O$ is symplectic, $O^TJO = J$, and $J^{-1} = -J$, the inverse of $O$ is given by
    \begin{align*}
        -JO^T J = \begin{pmatrix}
            P^T & -Q^T \\ Q^T & P^T
        \end{pmatrix} = O^T,
    \end{align*}
    so $O$ is orthogonal. Using the fact that $U_O$ is \related to $O,$ and then the expression for $\langle g_0, \rho(z) g_0 \rangle$ above, we now find that
    \begin{align*}
        \langle U_O g_0, \rho(z) U_O g_0 \rangle = \langle g_0, \rho(O^{-1}z) g_0 \rangle = \langle g_0, \rho(z) g_0 \rangle.
    \end{align*}
    Evaluating this at $z = (x, \omega)$, we get
    \begin{align*}
        \int_{\Rd} U_O g_0(t) \overline{T_x U_Og_0(t)} e^{-2\pi i \omega \cdot t} \, dt  = \int_{\Rd} g_0(t) \overline{T_x g_0(t)} e^{-2\pi i \omega \cdot t} \, dt.
    \end{align*}
    This shows that, for every $x \in \Rd$, the Fourier transform of the $L^1$-function $U_Og_0 \overline{T_xU_Og_0} - g_0 \overline{T_x g_0}$ vanishes. By the injectivity of the Fourier transform,
    \begin{align*}
        U_Og_0 \overline{T_xU_Og_0} = g_0 \overline{T_x g_0} \quad \text{for all } x \in \Rd.
    \end{align*}
    Taking $x = 0$ shows that $U_O g_0 = \alpha g_0$ for some function $\alpha \from \Rd \to \mathbb C$ of absolute value one, and then $\alpha = U_O g_0 / g_0 = T_x(\overline{g_0/U_O g_0}) = T_x(\alpha)$ shows that $\alpha$ is constant.
\end{proof}

\begin{remark}[Symplectic Rotations]
    In the proof of Theorem \ref{Gaussian_Gabor_frames_v2}, we used the fact that (up to a phase factor), we have $U g_0 = g_0$ for any unitary operator $U$ that is $\rho$-related to a symplectic matrix of the form
    \begin{align*}
        O = \begin{pmatrix}
            P & Q \\ - Q & P
        \end{pmatrix}.
    \end{align*}
    The group of such symplectic matrices is referred to as a unitary group and denoted by $\UtdR$, because the map $U(d, \mathbb C) \to \UtdR$ given by $P + i Q \mapsto \left(\begin{smallmatrix} P & Q \\ -Q & P \end{smallmatrix}\right)$ is an isomorphism of groups. It turns out that
    \begin{align*}
        U(2d, \mathbb R) = \SptdR \cap O(2d, \mathbb R),
    \end{align*}
    where $O(2d, \mathbb R)$ is the orthogonal group. Matrices in $\UtdR$ are therefore also referred to as \textit{symplectic rotations}. Again, we refer to de Gosson \cite{deGosson_book, deGosson_QM} for more details.

    The invariance of $g_0$ under operators $\rho$-related to symplectic rotations means that (just by symplectic covariance and Theorem \ref{Gaussian_Gabor_frames})
    $g_0$ generates a Gabor frame over any lattice of the form
    \begin{align*}
        O(D_\alpha \Zd \times D_\beta \Zd), \quad &\text{where } O \in \UtdR, D_\alpha = \mathrm{diag}(\alpha_1, \ldots, \alpha_d) \\ &\text{and } D_\beta = \mathrm{diag}(\beta_1, \ldots, \beta_d),
    \end{align*}
    if and only if $|\alpha_j \beta_j| < 1$. In other words, when it comes to the standard Gaussian, we have ``unitary freedom'' in our choice of lattice. 
    
    This unitary freedom is already included in Theorem \ref{Gaussian_Gabor_frames_v2}. To see this, let us choose
    \begin{align}
        \label{eq_A_for_Gaussian}
        A = \begin{pmatrix}
            P & Q \\ - Q & P
        \end{pmatrix} \begin{pmatrix}
            D_\alpha & 0 \\ 0 & D_\beta
        \end{pmatrix}, \quad \text{with } O = \begin{pmatrix}
            P & Q \\ - Q & P
        \end{pmatrix} \in \UtdR
    \end{align}
    and $D_\alpha$ and $D_\beta$ diagonal. We find that
    \begin{align*}
        X &= (P(DD_\alpha)^2 P^T + Q (DD_\alpha)^{-2} Q^T )^{-1}
        \\ \text{and} \quad Y &= -( -Q(DD_\alpha)^2 P^T + P(DD_\alpha)^{-2} Q^T) X,
    \end{align*}
    where we have used that $K = D_\alpha D_\beta$, which in turn follows from
    \begin{align*}
        A^T J A = \begin{pmatrix}
            D_\alpha & 0 \\ 0 & D_\beta
        \end{pmatrix}^T O^T J O \begin{pmatrix}
            D_\alpha & 0 \\ 0 & D_\beta
        \end{pmatrix} = \begin{pmatrix}
            0 & D_\alpha D_\beta \\ - D_\alpha D_\beta & 0
        \end{pmatrix}.
    \end{align*}
    If we choose $D = D_\alpha^{-1}$, then 
    \begin{align*}
        X = (PP^T + QQ^T)^{-1} \quad \text{and} \quad Y = -(-QP^T + PQ^T)X.
    \end{align*}
    If one writes out the equation $OO^T = I$, one finds that $X = I$ and $Y = 0$. Theorem \ref{Gaussian_Gabor_frames_v2} then tells us that $g_0$ generates a Gabor frame over $A\Ztd$ if and only if every entry of $K = D_\alpha D_\beta$ has absolute value $< 1$. 

    This example is another illustration of the utility of using the entries of the symplectic form $A^T J A$ as parameters. Given a lattice matrix $A$ of the form \eqref{eq_A_for_Gaussian}, without knowing the explicit decomposition into $O$, $D_\alpha$ and $D_\beta$, it seems very difficult to make sense of the condition $|\alpha_j \beta_j| < 1$ without considering the symplectic form $A^T J A$.
\end{remark}

Gaussian Gabor frames have also recently been investigated using methods from Kähler geometry by the second named author, Testorf and Wang \cite{Luef_Testorf_Wang}. It is interesting to note that the entries of $A^T J A$ appear in their results as well, where they are subject to conditions that are very different from those of Theorem \ref{Gaussian_Gabor_frames_v2}, such as conditions of linear independence over the integers.

\section{On the Necessity of Symplecticity}
\label{sec_necessity}

In this section, we explore the degree to which symplectic transformations are the only linear transformations of the time-frequency plane that preserve the structure of Gabor systems, in the sense of Corollary \ref{corollary_main}.

As a preliminary result, we observe that preservation of duality (at the level of lattice matrices) and the intertwining of symmetrized time-frequency shifts are both equivalent to symplecticity.

\begin{prop}
Let $\T \in \GLtdR$. Then, the following conditions are equivalent. 
\begin{itemize}
    \item[(a)] $\T$ is symplectic. 
    \item[(b)] $(\T A)^\circ = \T A^\circ$ for some (and if so, for all) $A \in \GLtdR$.
    \item[(c)] There exists a bounded invertible operator $\F$ on $L^2(\Rd)$ such that 
    \begin{align*}
        \F \rho(Ak) \F^{-1} = \rho(\T Ak)
    \end{align*}
    for all $A \in \GLtdR$ and $k \in \Ztd$.
\end{itemize}
\end{prop}
\begin{proof}
    Writing out the equation $(\T A)^\circ = \T A^\circ$, we get $-J \T^{-T} A^{-T} = - \T J A^{-T}$. Canceling $A^{-T}$ and inverting both sides, we see that this is equivalent to $J = \T^T J \T$, which is precisely the condition that $\T$ be symplectic. This proves that $(a) \iff (b)$ (as well as the claim in $(b)$, since $A$ cancels out).
    
    We know that $(a) \implies (c)$ by Proposition \ref{prop_metaplectic_ops}. If $(c)$ holds, then $\rho(Ak)$ and $\rho(\T Ak)$ must satisfy the same commutation relations for all $A \in \GLtdR$ and $k \in \Ztd$. Considering Equation \eqref{eq_final_commutation_relation} (which also holds for the symmetrized time-frequency shifts), this means that $A^T J A - (\T A)^T J (\T A)$ is integer valued. This implies that we have a well-defined map
    \begin{align*}
        \GLtdR &\to \MtdZ \\
        A &\mapsto A^T J A - (\T A)^T J (\T A).
    \end{align*}
    This map is clearly continuous, so since $\MtdZ$ is discrete, it must be constant on the identity component of $\GLtdR$. In particular, equating the outputs for $A = I$ and $A = 2I$, we get $J - \T^T J \T  = 4(J - \T^T J \T)$, so that $J = \T^T J \T$.
\end{proof}

We now wish to prove a similar result at the level of frame operators. That is, we wish to show that the conclusions of Corollary $\ref{corollary_main}$ (with $S \in \GLtdR$ arbitrary) imply that $S$ is symplectic. This turns out to be significantly more difficult. The reason is that frame operators are much more complicated objects than time-frequency shifts. In particular, they do not satisfy a nice commutation relation that we can exploit. Our approach will be to express certain time-frequency shifts as frame operators via the Janssen representation, and then modify the ideas that went into the proof of the previous proposition. This method of proof will not be able to distinguish between linear and antilinear operators (to be defined shortly), so we are forced to deal with antilinear operators as well. On the plus side, allowing antilinear operators gives us another symmetry of Gabor structures, in addition to symplectic covariance, namely complex conjugation (see Remark \ref{remark_complex_conjugation}). Moreover, our result (Theorem \ref{theorem_converse}) implies that this is essentially the only new symmetry that appears (see the discussion following Theorem \ref{theorem_converse} for details).

A function $F \from L^2(\Rd) \to L^2(\Rd)$ is called an \textit{antilinear operator} if $F(\alpha f + \beta g) = \overline{\alpha} Ff + \overline{\beta} Fg$ for all $f, g \in L^2(\Rd)$ and $\alpha, \beta \in \mathbb C$. Boundedness of antilinear operators is defined in precisely the same way as for linear operators. The adjoint $F^*$ of an antilinear operator is defined by the condition $\langle Ff, g \rangle = \overline{\langle f, F^*g \rangle}$ (and is guaranteed to exist by the usual argument if $F$ is bounded). An antilinear operator $F$ is called \textit{antiunitary} if it is invertible and satisfies $F^{-1} = F^*,$ which is equivalent to invertibility along with the condition $\langle Ff, Fg \rangle = \overline{\langle f, g \rangle}$ for all $f, g \in L^2(\Rd)$ (which shows that antiunitary operators are bounded, just like unitary operators).\footnote{Antilinear and antiunitary operators are really just linear and unitary operators $L^2(\Rd) \to \overline{L^2(\Rd)}$, where $\overline{L^2(\Rd)}$ is the set $L^2(\Rd)$ with the usual addition, but with scalar multiplication and inner product defined by $(\alpha f)(t) = \overline{\alpha} f(t)$ and $\langle f, g \rangle_{\overline{L^2(\Rd)}} = \overline{\langle f, g \rangle_{L^2(\Rd)}}$.}

\begin{remark}[Complex Conjugation]
    \label{remark_complex_conjugation}
    Let
    \begin{align*}
        \R \coloneqq \begin{pmatrix}
            I & 0 \\ 0 & -I
        \end{pmatrix}, \quad \text{so that } \R^T J \R = -J,
    \end{align*}
    and let $C \from L^2(\Rd) \to L^2(\Rd)$ denote complex conjugation:\ $Cf = \bar f$. This is an antiunitary operator satisfying $C^* = C^{-1} = C$. We claim that, for all lattices $\Lambda \subset \Rtd$ and $g, h \in L^2(\Rd)$, the following holds:\
    \begin{itemize}
        \item[(i)] $\mathcal G(Cg, \R \Lambda)$ is a Bessel sequence if and only if $\mathcal G(g, \Lambda)$ is a Bessel sequence;
        \item[(ii)] $C S_{g, h}^\Lambda C^* = S_{Cg, Ch}^{\R \Lambda}$ whenever $\mathcal G(g, \Lambda)$ and $\mathcal G(h, \Lambda)$ are Bessel sequences. 
    \end{itemize}
    To see this, note that $T_x C = C T_x$ and $M_\omega C = C M_{-\omega}$, so that 
    \begin{align*}
        \pi(\R z) C = M_{-\omega} T_x C = C \pi(z) \quad \text{for all } z = (x, \omega) \in \Rtd.
    \end{align*}
    The claim now follows by calculations that are identical to those in the proof of Theorem \ref{theorem_main}. Extending our language somewhat, this identity shows that $C$ is an antiunitary operator that is $\pi$-related to the (non-symplectic) matrix $R$, and our claim is that $C$ implements an antiunitary equivalence between the Gabor structures over $\Lambda$ and $R\Lambda$, for any lattice $\Lambda$. Thus, if we also consider antilinear equivalences, we have found another symmetry of Gabor structures. 
\end{remark}

The following theorem shows that complex conjugation and unitary operators \related to symplectic matrices (which we may take to be metaplectic) capture all the equivalences of Gabor structures stemming from linear transformations of the time-frequency plane. Note that assumptions $(i)$ and $(ii)$ are precisely the assumption that $\F$ implements an equivalence between the Gabor structures over $\Lambda$ and $\T \Lambda$ for every lattice $\Lambda \subset \Rtd$ (see Definition \ref{definition_main}).\footnote{To be completely precise, we would need to replace ``linear operator'' in Definition \ref{definition_main} by ``linear or antilinear operator''.} (It is tempting to call such an equivalence a \textit{global} or \textit{uniform} equivalence of Gabor structures.) 

\begin{theorem}[On the Necessity of Symplecticity]
    \label{theorem_converse}
    Let $\T \in \GLtdR$ and suppose that there exists a linear or antilinear bounded invertible operator $\F$ on $L^2(\Rd)$ such that, for all lattices $\Lambda \subset \Rtd$ and all $g, h \in L^2(\Rd)$, the following holds:
    \begin{itemize}
        \item[(i)] $\mathcal G(\F g, \T \Lambda)$ is a Bessel sequence if and only if $\mathcal G(g, \Lambda)$ is a Bessel sequence;
        \item[(ii)] $\F S_{g, h}^\Lambda \F^{-1} = S_{\F g, \F h}^{\T \Lambda}$ whenever $\mathcal G(g, \Lambda)$ and $\mathcal G(h, \Lambda)$ are Bessel sequences.
    \end{itemize}
    Then, either $\T$ or $\R \T$ is symplectic.
\end{theorem}

\begin{proof}
    First of all, we claim that $(i)$ and $(ii)$ together imply that $|\det \T| = 1$. To see this, recall that $(i)$ and $(ii)$ together imply that, for any atom $g \in L^2(\Rd)$ and lattice $\Lambda$, the Gabor system $\mathcal G(Fg, T\Lambda)$ is a frame if and only if $\mathcal G(g, \Lambda)$ is a frame (see the discussion following Definition \ref{definition_main}). Choose now any atom $g \in L^2(\Rd)$ and lattice $\Lambda$ such that $\mathcal G(g, \Lambda)$ is a frame. By induction, $\mathcal G(F^n g, T^n \Lambda)$ is a frame for every integer $n \geq 1$. By the density theorem (Theorem \ref{theorem_density}), we must have $|\det T|^n |\Lambda| = |T^n \Lambda| \leq 1$ for all $n$, hence $|\det T| \leq 1$. Now, $\mathcal G(F^{-1}g, T^{-1}\Lambda)$ also has to be a frame, because $\mathcal G(g, \Lambda) = \mathcal G(FF^{-1}g, TT^{-1} \Lambda)$ is a frame. Thus, in the same manner, induction leads to the conclusion that $|\det T^{-1}| \leq 1$, so $|\det T| = 1$.
    
    Fix now any $0 \neq g \in \Falg$ and $A \in \GLtdR$. By point $(iv)$ of Lemma \ref{lemma_properties_of_Feichtinger}, there exists $\epsilon > 0$ such that $g$ generates a Gabor frame over $\epsilon A \Ztd$ and the canonical dual atom $h \coloneqq (S_{g, g}^{\epsilon A})^{-1} g$ is in $\Falg$ as well. We will write $B \coloneqq \epsilon A$, so we have that $\mathcal G(g, B \Ztd)$ is a Gabor frame and $h = (S_{g, g}^B)^{-1} g \in \Falg$.
    
    By the Wexler-Raz biorthogonality relations (Proposition \ref{prop_Wexler}), we have that
    \begin{align*}
        \blangle \pi(B^\circ l) h, \pi(B^\circ k) g \brangle = |B| \delta_{lk} \quad \text{for all } k, l \in \Ztd
    \end{align*}
    and hence, by Janssen's representation (Proposition \ref{prop_Janssen}), 
    \begin{align*}
        S^B_{g, \pi( B^\circ l ) h} = \frac{1}{|B|}\sum_{k \in \Ztd} \blangle \pi(B^\circ l) h, \pi(B^\circ k) g \brangle \pi(B^\circ k) = \pi(B^\circ l) \quad \text{for all } l \in \Ztd.
    \end{align*}
    In what follows, we will alternate between using assumption $(ii)$ and Janssen's representation until we arrive at
    \begin{align}
        \label{eq_proof_necessity}
        \pi (B^\circ l)(f) = S^{(\T^{-1}(\T B)^\circ)^\circ}_{g, \pi(B^\circ l) h}(f) \quad \text{for all } f \in \F^{-1}(\Falg).
    \end{align}
    
    Let $f \in \F^{-1}(\Falg)$. We have shown that $\pi(B^\circ l) = S^B_{g, \pi( B^\circ l ) h}$. Using assumption $(ii)$ and then Janssen's representation, we find that
    \begin{align*}
        \pi (B^\circ l)(f) &= S^{B}_{g, \pi(B^\circ l)h} (f) = \F^{-1} S^{\T B}_{\F g, \F \pi(B^\circ l) h} (\F f) = \frac{1}{|\T B|} \F^{-1} S^{(\T B)^\circ}_{\F g, \F f} \bl( \F \pi( B^\circ l) h \br).        
    \end{align*}
    The application of Janssen's representation is valid since $Fg, F\pi(B^\circ l)h$ and $Ff$ all generate Bessel sequences over $TB \Ztd$. Indeed, $Fg$ and $F\pi(B^\circ l)h$ generate Bessel sequences over $TB \Ztd$ by assumption $(i)$, since $g, \pi(B^\circ l)h \in \Falg$ generate Bessel sequences over any lattice by Lemma \ref{lemma_properties_of_Feichtinger} (and in particular over $B\Ztd$), and $Ff \in \Falg$ also generates a Bessel sequence over any lattice. 
    
    We now apply assumption (ii) again (in the opposite direction this time, using $(TB)^\circ = TT^{-1}(TB)^\circ$) to obtain
    \begin{align*}
        \frac{1}{|\T B|} \F^{-1} S^{(\T B)^\circ}_{\F g, \F f} \bl( \F \pi( B^\circ l) h \br) 
        = \frac{1}{|\T B|} S^{\T^{-1}(\T B)^\circ}_{g, f} \bl( \pi(B^\circ l) h \br).
    \end{align*}
    Applying Janssen's representation one more time, we get
    \begin{align*}
        \frac{1}{|\T B|} S^{\T^{-1}(\T B)^\circ}_{g, f} \bl( \pi(B^\circ l) h \br)  
        = \frac{1}{|\T B|} \frac{1}{|\T^{-1}(\T B)^\circ|} S^{(\T^{-1}(\T B)^\circ)^\circ}_{g, \pi(B^\circ l) h}(f).
    \end{align*}
    This time, it is justified by the fact that $g, \pi(B^\circ l)h \in \Falg$ generate Bessel sequences over any lattice and $f$ generates a Bessel sequence over $\T^{-1}(\T B)^\circ \Ztd$ by assumption $(i)$ (since $\F f \in \Falg$ generates a Bessel sequence over any lattice, and in particular over $\T \T^{-1}(\T B)^\circ \Ztd$). Finally, the constant we have ended up with in front of the frame operator equals one, because $|B^\circ| = |B|^{-1}$ and $|\det T| = 1$. Stringing together the last three displayed equations, we have proved \eqref{eq_proof_necessity}.
    
    Since $\Falg$ is dense in $L^2(\Rd)$ and since bounded and invertible operators map dense sets to dense sets, \eqref{eq_proof_necessity} implies that
    \begin{align*}
        \pi (B^\circ l) = S^{(\T^{-1}(\T B)^\circ)^\circ }_{g, \pi(B^\circ l)h} \quad \text{for all } l \in \Ztd.
    \end{align*}
    Frame operators over any given lattice commute with time-frequency shifts over the same lattice, so this equality implies that $\pi(B^\circ l)$ commutes with all time-frequency shifts over $(\T^{-1}(\T B)^\circ)^\circ$. Thus, we have
    \begin{align*}
        B^\circ l \in (\T^{-1}(\T B)^\circ)^{\circ \circ} \Ztd = \T^{-1}(\T B)^\circ \Ztd \quad \text{for all } l \in \Ztd.
    \end{align*}
    
    We have now shown that $B^\circ \Ztd \subset \T^{-1}(\T B)^\circ \Ztd$, which means that the matrix
    \begin{align*}
        (\T^{-1}(\T B)^\circ)^{-1} B^\circ = ((\T B)^\circ)^{-1} \T B^\circ
    \end{align*}
    is integer valued. Recalling that $B = \epsilon A$ and using the definition of $^\circ$, we find that
    \begin{align*}
        ((\T B)^\circ)^{-1} \T B^\circ = - A^T(\T^T J \T)J A^{-T} \in \MtdZ.
    \end{align*}
    But $A \in \GLtdR$ was arbitrary, and the map
    \begin{align*}
        \GLtdR &\to \MtdZ
        \\ A &\mapsto A^T (\T^T J \T) J A^{-T}
    \end{align*}
    is continuous, so it must be constant on the identity component $\mathrm{GL}^+(2d, \mathbb R)$ of $\GLtdR$, which consists of all invertible matrices with positive determinants. Equating the outputs for an arbitrary $A \in \mathrm{GL}^+(2d, \mathbb R)$ and for $A = I$, we get
    \begin{align*}
        A^T (\T^T J \T) J A^{-T} = (\T^T J \T) J \quad \text{for all } A \in \mathrm{GL}^+(2d, \mathbb R),
    \end{align*}
    which implies that $(\T^T J \T) J$ commutes with all of $\mathrm{GL}^+(2d, \mathbb R)$. The only matrices with this property are scalar multiples of the identity,\footnote{Here is a quick argument showing that any $M \in \MtdR$ that commutes with all of $\mathrm{GL}^+(2d, \mathbb R)$ is a scalar multiple of the identity. Let $E^{ij} \in \MtdR$ be the matrix with $1$ in the $(i, j)$-th position and zeros everywhere else. Then $I + E^{ij}$ has positive determinant, so $M(I + E^{ij}) = (I + E^{ij})M$, and hence $ME^{ij} = E^{ij}M$. Calculating these products leads to the conclusion that $M$ is a scalar multiple of the identity.} so there is some $r \in \mathbb R$ such that $(\T^T J \T) J = r I$. Since $| \det \T |$ and  $|\det J|$ both equal one, we must have $r = \pm 1$. If $r = -1$, then $\T$ is symplectic, while if $r = 1$, then $\R \T$ is symplectic (because then $T^TJT = -J = R^TJR$).
\end{proof}

\begin{remark}
    As $\R$ corresponds to complex conjugation, which is antilinear, it seems reasonable to suspect that $\T$ has to be symplectic if $\F$ is assumed linear. This is equivalent to the statement that there is no \textit{linear} bounded operator which implements an equivalence between the Gabor structures over $\Lambda$ and $R\Lambda$ for every lattice $\Lambda$. Settling this question would give a complete characterization of all linear transformations $T$ of the time-frequency plane which lift to \textit{linear} bounded operators implementing an equivalence between the Gabor structures over $\Lambda$ and $T\Lambda$ for every lattice $\Lambda$; the answer is either $\SptdR$ or $\SptdR \cup R \SptdR$.
\end{remark}

Theorem \ref{theorem_converse} does not allow us to conclude that $\F$ or $C \F$ must be \related to a symplectic matrix (where $C$ is complex conjugation—see Remark \ref{remark_complex_conjugation}). However, combined with Theorem \ref{theorem_main}, it shows that the equivalence between the Gabor structures over $\Lambda$ and $\T \Lambda$ that is implemented by $\F$ or $C\F$ is also implemented by some unitary operator $U$ that is $\rho$-related to $T$ or $RT$. In fact, suppose for now that $T$ is symplectic and let $U$ be a unitary operator \related to $\T$. Setting $W \coloneqq U^* \F$, we find that
\begin{align}
    \label{eq_W}
    W S_{g, h}^\Lambda W^{-1} = S_{Wg, Wh}^\Lambda 
\end{align} 
for all lattices $\Lambda \subset \Rtd$ and $g, h \in L^2(\Rd)$ generating Bessel sequences over $\Lambda$. This shows that $\F$ ($= UW$) is just the operator $U$ up to some operator $W$ that intertwines frame operators over any given lattice. In more generic (and imprecise) terminology, $W$ is like an automorphism of the Gabor structure over $\Lambda$, and the isomorphism $\F$ from the Gabor structure over $\Lambda$ to the Gabor structure over $\T \Lambda$ is just this automorphism followed by the isomorphism $U$. If instead $RT$ is symplectic, we let $U$ be $\rho$-related to $RT$ and we set $W \coloneqq U^* CF$, then $W$ still satisfies \eqref{eq_W}, so a similar statement regarding $CF$ ($= UW$) is true in this case. It would be interesting to know whether there are any operators $W$ satisfying \eqref{eq_W} besides scalar multiples of the identity. If not, then $F$ or $CF$ would have to be $\rho$-related to $T$ or $RT$ (and hence a scalar multiple of a metaplectic operator).

In summary, whenever a linear transformation $T$ of $\Rtd$ lifts to a bounded (linear or antilinear) operator $F$ on $L^2(\Rd)$ implementing an equivalence between the Gabor structures over $\Lambda$ and $T\Lambda$, for all lattices $\Lambda$, then either $T$ or $RT$ is symplectic and we can find a metaplectic operator $U$ such that $U$ or $CU$ implement the same equivalences as $F$. Thus, symplectic covariance and Remark \ref{remark_complex_conjugation} capture all symmetries of Gabor structures stemming from linear transformations of the time-frequency plane.

\section{Separable Lattices and Lagrangian Planes}
\label{subsec_separable_geometric}

In this section, we show that separable lattices can be characterized in terms of a standard notion from symplectic geometry, namely transversal pairs of Lagrangian planes.

Given an antisymmetric matrix $\theta \in \GLtdR$, consider the \textit{$\theta$-symplectic group} 
\begin{align*}
    \SpthetatdR \coloneqq \{ S \in \GLtdR : S^T \theta S = \theta \}.
\end{align*}
This is precisely the group of linear transformations that preserve the symplectic form $\Omega$ represented by $\theta$. In this terminology, ordinary symplectic matrices are $J$-symplectic and $\SptdR = \mathrm{Sp}_J(2d, \mathbb R)$.

Consider the subspaces 
\begin{align}
    \label{eq_std_Lagrangian_frame}
    \ell_x \coloneqq \{ (x, 0) \in \Rtd : x \in \Rd \} \quad \text{and} \quad \ell_\omega \coloneqq \{ (0, \omega) \in \Rtd : \omega \in \Rd \}
\end{align}
of $\Rtd$. The standard symplectic form $\sigma$ determines an isomorphism $\ell_\omega \cong \ell_x^*$ of vector spaces (where $\ell_x^*$ is the dual of $\ell_x$) via $(0,\omega) \mapsto \sigma((0,\omega),-)$, and if we denote the resulting duality pairing by $\langle -, - \rangle_\sigma$, then we can write
\begin{align*}
    \sigma \bl((x, \omega), (y, \eta) \br) &= \sigma \bl((0, \omega), (y, 0) \br) - \sigma \bl((0, \eta), (x, 0) \br) = \langle y, \omega \rangle_\sigma - \langle x, \eta \rangle_\sigma,
\end{align*}
so that the symplectic form is just the antisymmetrized duality pairing (cf.\ Equation \eqref{std_symp_form}). In fact, this shows how to naturally define a symplectic form on $V \oplus V^*$ for any vector space $V$, and, if one phrases the theory of Gabor frames in more abstract terms (as one does when considering general locally compact abelian groups), then this is essentially how the symplectic form enters the commutation relations in the first place.

Note that $\sigma$ vanishes identically when restricted to either $\ell_x$ or $\ell_\omega$. Whenever we have a direct sum decomposition of $\Rtd$ into two such subspaces, we can decompose $\sigma$ in this manner. This also works for an arbitrary symplectic form, which motivates the following definition. Further details can be found in de Gosson's monograph \cite{deGosson_book}.

\begin{define}[Lagrangian Planes]
    Let $\Omega \from \Rtd \times \Rtd \to \mathbb R$ be a symplectic form. An \textit{$\Omega$-Lagrangian plane} is a $d$-dimensional subspace $\ell \subset \Rtd$ on which $\Omega$ vanishes identically, i.e.\ $\Omega|_{\ell \times \ell} = 0$. A pair $(\ell, \ell')$ of $\Omega$-Lagrangian planes is called \textit{transversal} if $\ell + \ell' = \Rtd$ (which is equivalent to $\Rtd = \ell \oplus \ell'$ because of the dimensions involved).
\end{define}

\noindent In this terminology, the discussion preceding the definition says that if we have a transversal pair $(\ell, \ell')$ of $\Omega$-Lagrangian planes, then we obtain an isomorphism $\Rtd \cong \ell \oplus \ell^*$ under which $\Omega$ becomes the antisymmetrized duality pairing. If $\theta$ is the matrix representing $\Omega$, we will also refer to $\Omega$-Lagrangian planes as $\theta$-Lagrangian.

We now show that lattices symplectically related to separable lattices can be characterized geometrically in terms of this notion. The closely related notion of Lagrangian subgroups of $\Rtd$, and its role in time-frequency analysis, has recently been investigated by Fulsche and Rodriguez Rodriguez \cite{fulsche2023commutative}. Recall that $\Theta_\Lambda$ denotes the set of all symplectic forms determined by the lattice $\Lambda$ (see Equation \eqref{eq_set_of_symp_forms}).

\begin{prop}[Separable Lattices and Lagrangian Planes]
    \label{prop_separable_lattices_geometric}
    Let $\Lambda \subset \Rtd$ be a lattice. Then, the following conditions are equivalent to $\Lambda$ being symplectically related to a separable lattice (and hence can be added to the list of equivalent conditions in Proposition \ref{prop_separable_latties}).
    \begin{itemize}
        \item[(c)] There exists a basis $(\lambda_1, \ldots, \lambda_{2d})$ for $\Lambda$ (as a $\mathbb Z$-module) and a transversal pair $(\ell, \ell')$ of $J$-Lagrangian planes such that 
        \begin{align*}
            \{ \lambda_1, \ldots, \lambda_{2d} \} \subset \ell \cup \ell'.
        \end{align*}
        \item[(d)] For some (and if so, for every) $\theta \in \Theta_\Lambda$, there exists a basis $(b_1, \ldots, b_{2d})$ for $\Ztd$ and a transversal pair $(\ell_\theta, \ell_\theta')$ of $\theta$-Lagrangian planes such that
        \begin{align*}
            \{ b_1, \ldots, b_{2d} \} \subset \ell_\theta \cup \ell_\theta'.
        \end{align*}
    \end{itemize}
\end{prop}
\begin{proof}
    Condition $(a)$ refers to Proposition \ref{prop_separable_latties}. We will show that $(c) \iff (d)$ and $(a) \iff (c)$.

    The equivalence $(c) \iff (d)$ amounts to a simple change of basis. If $A$ is any lattice matrix for $\Lambda$, then $\theta \coloneqq A^T J A \in \Theta_\Lambda$ and we have that
    \begin{align*}
        w^T \theta z = (Aw)^T J (Az) \quad \text{for all } z, w \in \Rtd.
    \end{align*}
    This identity immediately implies that a subspace $\ell_\theta \subset \Rtd$ is a $\theta$-Lagrangian plane if and only if $A(\ell_\theta)$ is a $J$-Lagrangian plane. Being invertible, $A$ preserves direct sums, so $(\ell_\theta, \ell_\theta')$ is a transversal pair of $\theta$-Lagrangian planes if and only if $(A(\ell_\theta), A(\ell'_\theta))$ is a transversal pair of $J$-Lagrangian planes. Finally, since $\Lambda = A \Ztd$, $(b_1, \ldots, b_{2d})$ is a basis for $\Ztd$ if and only if $(Ab_1, \ldots, Ab_{2d})$ is a basis for $\Lambda$. Since $A$ was an arbitrary lattice matrix for $\Lambda$, this shows that $(c)$ is equivalent to $(d)$ for an arbitrary $\theta \in \Theta_\Lambda$.

    If $(a)$ holds, i.e.\ if there is $A \in \GLtdR$ such that 
    \begin{align*}
       \Lambda = A\Ztd \quad \text{and} \quad A^T J A = \begin{pmatrix} 0 & K \\ - K^T & 0 \end{pmatrix} \text{ for some } K \in \GLdR, 
    \end{align*}
    then $(c)$ holds with $\lambda_1, \ldots, \lambda_{2d}$ equal to the columns of $A$ (in order from left to right) and
    \begin{align*}
        \ell &\coloneqq A(\ell_x) = \mathrm{span}_\mathbb R \{\lambda_1, \ldots, \lambda_d\} \\
        \text{and} \quad \ell' &\coloneqq A(\ell_\omega) = \mathrm{span}_\mathbb R \{\lambda_{d+1}, \ldots, \lambda_{2d}\}.
    \end{align*}
    Indeed, the condition that the diagonal blocks of $A^T J A$ vanish is precisely the condition that the standard symplectic form vanishes on these subspaces, so that $(\ell, \ell')$ is a transversal pair of $J$-Lagrangian planes.

    Conversely, if $(c)$ holds, then we can re-index our basis elements $\lambda_j$ so that
    \begin{align*}
        \ell = \mathrm{span}_\mathbb R \{ \lambda_1, \ldots, \lambda_d \} \quad \text{and} \quad \ell' = \mathrm{span}_\mathbb R \{ \lambda_{d+1}, \ldots, \lambda_{2d} \}.
    \end{align*}
    Taking $(\lambda_1, \ldots, \lambda_{2d})$ as the columns of $A$ implies that $\Lambda = A\Ztd$, and the fact that $\ell$ and $\ell'$ are $J$-Lagrangian planes implies that the diagonal blocks of $A^T J A$ vanish, i.e.\ $A^T J A = \left(\begin{smallmatrix} 0 & K \\ K' & 0 \end{smallmatrix}\right)$ for some $K, K' \in \GLdR$. Since $A^T J A$ is antisymmetric, we have $K' = -K^T$, so $(a)$ holds.
\end{proof}

\appendix

\section{Explicit Expressions for the $\rho$-Related Unitary Operators}
\label{appendix}

In this appendix, we explicitly construct all the unitary operators that are \related to a large class of symplectic matrices, namely the \textit{free} symplectic matrices. This is an adaptation of the presentation of metaplectic operators in de Gosson \cite{deGosson_book}. With $S \in \SptdR$ written in block form, 
\begin{align}
    \label{block_form}
    S = \begin{pmatrix}
        A & B \\ C & D
    \end{pmatrix} \quad \text{where } A, B, C, D \in \MdR,
\end{align}
we say that $S$ is \textit{free} if $\det B \neq 0$.

\begin{prop}
    \label{prop_free_metaplectic_ops}
    Let $S \in \SptdR$ be a free symplectic matrix written in block form, as in Equation \eqref{block_form}. Define the quadratic form
    \begin{align*}
        W(t, x) = \frac 12  t^T DB^{-1} t - x^T B^{-1} t + \frac 12 x^T B^{-1}A x \quad (t, x \in \Rd).
    \end{align*}
    Then, the unitary operator $U$ on $L^2(\Rd)$ defined by
    \begin{align}
        \label{eq_quadratic_Foruier} 
        U f(t) = \sqrt{|\det B^{-1}|} \int_{\Rd} f(x) e^{2 \pi i W(t, x)} \, dx
    \end{align}
    is \related to $S$, meaning that $U \rho(z) U^* = \rho(Sz)$ for all $z \in \Rtd$. 
\end{prop}

We now wish to outline how to prove Proposition \ref{prop_free_metaplectic_ops} by using the generators of the symplectic group mentioned in Section \ref{prelim_symplectic}. Before we do so, we note the following factorization result, which together with Proposition \ref{prop_free_metaplectic_ops} implies that every symplectic matrix has a \related unitary operator (which moreover can be written as a composition of two operators with the form of $U$ in Equation \eqref{eq_quadratic_Foruier}, which are referred to as \textit{quadratic Fourier transforms}).

\begin{prop}
    \label{prop_two_free_symplectic}
    Every $S \in \SptdR$ can be written as a product of two free symplectic matrices.
\end{prop}
\begin{proof}
    See \cite[Theorem~60]{deGosson_book}. Further details can be found in \cite[Section~1.2.2]{masters}.
\end{proof}

Recall that we defined the symplectic matrices
\begin{align*}
    J \coloneqq \begin{pmatrix}
        0 & I \\ -I & 0 
    \end{pmatrix}, \quad V_P \coloneqq \begin{pmatrix}
        I & 0 \\ -P & I 
    \end{pmatrix} \quad \text{and} \quad M_L \coloneqq \begin{pmatrix}
        L^{-1} & 0 \\ 0 & L^T
    \end{pmatrix}
\end{align*}
for $P = P^T \in \MdR$ and $L \in \GLdR$, along with the unitary operators $U_J$, $U_{V_P}$ and $U_{M_L}$ on $L^2(\Rd)$ given by
\begin{align*}
    U_J f(t) &= \mathcal Ff(t)  = \int_{\Rd} f(x) e^{-2 \pi i x \cdot t} \, dx, \\
    U_{V_P} f(t) &= e^{-\pi i (Pt)\cdot t} f(t) \quad \text{and} \quad U_{M_L} f(t) = \sqrt{|\det L|} f(Lt).
\end{align*}

\begin{lemma}
    \label{appendix_lemma_generators}
    Let $S \in \SptdR$ be a free symplectic matrix written in block form, as in Equation \eqref{block_form}. Then, $DB^{-1}$ and $B^{-1}A$ are symmetric and 
    \begin{align*}
        S &= V_{-DB^{-1}} M_{B^{-1}} J V_{-B^{-1}A}.
        % \\ &= \begin{pmatrix} I & 0 \\ DB^{-1} & I \end{pmatrix} \begin{pmatrix} B & 0 \\ 0 & B^{-T} \end{pmatrix} \begin{pmatrix} 0 & I \\ -I & 0 \end{pmatrix} \begin{pmatrix} I & 0 \\ B^{-1} A & I \end{pmatrix}.
    \end{align*}
    Thus, by Proposition \ref{prop_two_free_symplectic}, the collection 
    \begin{align*}
        \{ J \} \cup \{ V_P : P = P^T \in \MdR \} \cup \{ M_L : L \in \GLdR \}
    \end{align*}
    generates the symplectic group $\SptdR$.
\end{lemma}

\begin{proof}
    Multiplying out shows that
    \begin{align}
        \label{temp_proof_1}
        S = \begin{pmatrix}
            I & 0 \\ DB^{-1} & I
        \end{pmatrix} \begin{pmatrix}
            B & 0 \\ 0 & DB^{-1}A - C 
        \end{pmatrix} \begin{pmatrix}
            0 & I \\ -I & 0
        \end{pmatrix} \begin{pmatrix}
            I & 0 \\ B^{-1} A & I
        \end{pmatrix}.
    \end{align}
    Using that $S^T J S = J$ and $J^{-1} = -J$, we find that
    \begin{align*}
        S^{-1} = -J S^T J = \begin{pmatrix}
            D^T & -B^T \\ -C^T & A^T
        \end{pmatrix}.
    \end{align*}
    The identity $S^{-1} S = I = S S^{-1}$ then implies that
    \begin{align*}
        BA^T - AB^T = 0 = D^TB - B^TD \quad \text{and} \quad DA^T - CB^T = I,
    \end{align*}
    from which we get
    \begin{align*}
        (B^{-1}A)^T &= B^{-1} (BA^T) B^{-T} = B^{-1}(AB^T)B^{-T} = B^{-1}A \\
        \text{and} \quad  (DB^{-1})^T &= B^{-T} (D^TB) B^{-1} = B^{-T}(B^TD) B^{-1} = DB^{-1},
    \end{align*}
    which proves that $B^{-1}A$ and $DB^{-1}$ are symmetric. We also find that
    \begin{align*}
        B^{-T} = (DA^T - CB^T)B^{-T} = D(B^{-1}A)^T - C = DB^{-1}A - C,
    \end{align*}
    which, upon insertion into Equation \eqref{temp_proof_1}, gives the result. 
\end{proof}

We now identify $U_J$, $U_{V_P}$ and $U_{M_L}$ as operators on $L^2(\Rd)$ which are $\rho$-related to $J$, $V_P$ and $M_L$, respectively.

\begin{prop}
    \label{appendix_prop_generators}
    The unitary operators $U_J$, $U_{V_P}$ and $U_{M_L}$ are \related to the generators from Lemma \ref{appendix_lemma_generators}, meaning that
    \begin{align*}
        U_J \rho(z) = \rho(Jz) U_J, \quad U_{V_P} \rho(z) = \rho(V_P z) U_{V_P} \quad \text{and} \quad U_{M_L} \rho(z) = \rho(M_L z) U_{M_L}
    \end{align*}
    for all $z \in \Rtd$. 
\end{prop}
\begin{proof}
    The identities in this proposition can be verified by applying both sides to an arbitrary $f \in L^2(\Rd)$ and evaluating the result at an arbitrary $t \in \Rd$. The calculations are straightforward but tedious. They can be found in \cite[Lemma~6.1.2]{masters}.
\end{proof}

\begin{proof}[Proof of Proposition \ref{prop_free_metaplectic_ops}]
    By Lemma \ref{appendix_lemma_generators} and Proposition \ref{appendix_prop_generators}, the unitary operator 
    \begin{align*}
        U \coloneqq U_{V_{-DB^{-1}}} U_{M_{B^{-1}}} U_J U_{V_{-B^{-1}A}}
    \end{align*}
    is \related to $S$. Using the definitions of these operators we find that, for any $f \in L^2(\Rd)$,
    \begin{align*}
         U_J U_{V_{-B^{-1}A}} f(t) = \int_{\Rd} e^{\pi i x \cdot B^{-1}A x} f(x) e^{-2\pi i x \cdot t} \, dx,
    \end{align*}
    and hence
    \begin{align*}
        Uf(t) = e^{\pi i t \cdot DB^{-1} t} \sqrt{|\det B^{-1}|} \int_{\Rd} e^{\pi i x \cdot B^{-1}A x} f(x) e^{-2\pi i x \cdot B^{-1} t} \, dx,
    \end{align*}
    which we recognize as the operator in Equation \eqref{eq_quadratic_Foruier}.
\end{proof}

\section{The Pre-Iwasawa Factorization}
\label{sec_pre_Iwasawa}

\begin{prop}[The Pre-Iwasawa Factorization]
\label{prop_pre_Iwasawa}
Write $S \in \SptdR$ in block form, 
\begin{align*}
    S = \begin{pmatrix}
        A & B \\ C & D
    \end{pmatrix} \quad \text{where } A, B, C, D \in \MdR.
\end{align*}
Then, $AA^T + BB^T$ is positive definite, and if we define
\begin{align*}
X \coloneqq (AA^T + BB^T)^{-1}, \quad Y \coloneqq -(CA^T + DB^T)X, \quad P \coloneqq X^{1/2}A \quad \text{and} \quad Q \coloneqq X^{1/2}B,
\end{align*}
we have
\begin{align}
\label{eq_pre_Iwasawa}
S = \begin{pmatrix} I & 0 \\ -Y & I \end{pmatrix} \begin{pmatrix} X^{-1/2} & 0 \\ 0 & X^{1/2} \end{pmatrix} \begin{pmatrix} P & Q \\ -Q & P \end{pmatrix},
\end{align}
where each factor is symplectic.
\end{prop}

\begin{proof}
$AA^T + BB^T$ is just the upper left $d \times d$ block of the trivially positive definite matrix $SS^T$. Considering the inner products $z^T SS^T z$ with $z = (x, 0) \in \Rtd$ straightforwardly leads to the conclusion that $AA^T + BB^T$ is positive definite as well. Thus, $X^{-1}$ and $X^{1/2}$ are well-defined.

We note that $BA^T - AB^T = 0$ and $DA^T - CB^T = I$. This is just the identity $I = SS^{-1}$ and the observation that $S^T J S = J \iff S^{-1} = -JS^TJ$ (a bit more detail can be found in the proof of Lemma \ref{appendix_lemma_generators}). These identities will be used at multiple points of this proof.

Writing $O \coloneqq \begin{pmatrix} P & Q \\ -Q & P \end{pmatrix} =  \begin{pmatrix} X^{1/2} & 0 \\ 0 & X^{1/2} \end{pmatrix} \begin{pmatrix} A & B \\ -B & A \end{pmatrix}$, we find that
\begin{align}
\label{eq_JOTJ}
-J O^T J =
    -J \begin{pmatrix} A^T & -B^T \\ B^T & A^T \end{pmatrix} J \begin{pmatrix} X^{1/2} & 0 \\ 0 & X^{1/2}
\end{pmatrix} = \begin{pmatrix} A^T & -B^T \\ B^T & A^T \end{pmatrix} \begin{pmatrix} X^{1/2} & 0 \\ 0 & X^{1/2} \end{pmatrix},
\end{align}
and consequently
\begin{align*}
    O(-JO^TJ) = \begin{pmatrix} X^{1/2} & 0 \\ 0 & X^{1/2} \end{pmatrix} \begin{pmatrix} AA^T + BB^T & -AB^T + BA^T \\ -BA^T + AB^T & BB^T + AA^T  \end{pmatrix} \begin{pmatrix} X^{1/2} & 0 \\ 0 & X^{1/2} \end{pmatrix} = I,
\end{align*}
This shows that $O^{-1} = -JO^TJ$, so $O$ is symplectic. It is straightforward to verify that the middle matrix in the decomposition \eqref{eq_pre_Iwasawa} is symplectic (it is just the symplectic matrix $M_{X^{1/2}}$, in the notation of Section \ref{prelim_symplectic}). Once we have shown that \eqref{eq_pre_Iwasawa} holds, the fact that $\SptdR$ is a group implies that the remaining matrix $\left(\begin{smallmatrix} I & 0 \\ -Y & I \end{smallmatrix}\right)$ is symplectic as well.

Multiplying both sides of \eqref{eq_pre_Iwasawa} by $O^{-1} = - JO^T J$ from the right (and using \eqref{eq_JOTJ} for the left hand side), we find that it is equivalent to
\begin{align*}
    \begin{pmatrix} A & B \\ C & D \end{pmatrix} \begin{pmatrix} A^T & -B^T \\ B^T & A^T \end{pmatrix} = \begin{pmatrix} I & 0 \\ -Y & I \end{pmatrix} \begin{pmatrix} X^{-1} & 0 \\ 0 & I \end{pmatrix},
\end{align*}
which is quickly verified by calculating the left hand side:\
\begin{align*}
    \begin{pmatrix} A & B \\ C & D \end{pmatrix} \begin{pmatrix} A^T & -B^T \\ B^T & A^T \end{pmatrix} = \begin{pmatrix}
        AA^T + BB^T & -AB^T + BA^T \\ CA^T + DB^T & -CB^T + DA^T
    \end{pmatrix} = \begin{pmatrix}
        X^{-1} & 0 \\ -YX^{-1} & I
    \end{pmatrix}.
\end{align*}
\end{proof}

\bibliographystyle{abbrv}

\end{document}